\documentclass[12pt]{article}
\usepackage{fancyhdr, array,calc,graphicx,url,tabularx}
\usepackage{geometry}
\usepackage{latexsym}
\usepackage{amssymb}
\usepackage{amsmath}
\usepackage{makeidx}
\usepackage{undertilde}

\makeindex

{
   \end{minipage}
   \vspace*{\stretch{3}}
   \clearpage
}

\renewcommand{\gg}{\gamma}

\newcommand{\bR}{{\bf P}}

\newcommand{\rest}{\restriction}
\newcommand{\la}{\langle}
\newcommand{\ra}{\rangle}

\newcommand{\card}[1]{{\vert #1 \vert} }

\newcommand{\forces}{\Vdash}

\renewcommand{\models}{\vDash}
\newcommand{\powerset}{{\cal P}}
\newcommand{\cp}{{\rm cp }}

\newcommand{\cf}{{\rm cf}}

\newcommand{\K}{{ k:V \to V }}

\newtheorem{theorem}{Theorem}[section]
\newtheorem{proposition}[theorem]{Proposition}
\newtheorem{definition}[theorem]{Definition}

\newtheorem{lemma}[theorem]{Lemma}
\newtheorem{corollary}[theorem]{Corollary}

\newtheorem{conjecture}[theorem]{Conjecture}

\numberwithin{figure}{section}

\newenvironment{proof}{{\it{
Proof.}}}{\nopagebreak\mbox{}{\hfill$\square$}
\par\bigskip}
\newenvironment{Claim}{\noindent{\it
Claim} {}}{}
\newenvironment{Case}{\noindent{\it
Case}}{}

\newcommand{\rthm}[1]{Theorem~\ref{#1}}
\newcommand{\rlem}[1]{Lemma~\ref{#1}}

\newcommand{\rcor}[1]{Corollary~\ref{#1}}
\newcommand{\rdef}[1]{Definition~\ref{#1}}

\newcommand{\rsec}[1]{Section~\ref{#1}}

\def\k{\kappa}
\def\a{\alpha}
\def\b{\beta}
\def\d{\delta}

\def\l{\lambda}

\def\P{{\mathcal{P} }}
\def\W{{\mathcal{W} }}
\def\Q{{\mathcal{ Q}}}

\def\K{{\mathcal{ K}}}

\def\R{{\mathcal R}}

\def\H{{\rm{HOD}}}
\def\M{{\mathcal{M}}}
\def\N{{\mathcal{N}}}

\def\T {{\mathcal{T}}}
\def\U{{\mathcal{U}}}
\def\S{{\mathcal{S}}}

\def\VT{{\vec{\mathcal{T}}}}
\def\VU{{\vec{\mathcal{U}}}}
\def\VS{{\vec{\mathcal{S}}}}
\def\VW{{\vec{\mathcal{W}}}}

\def\bR{{\mathbb{R}}}

\def\cp #1{{ crit  #1 }}
\def\card#1{\left|#1\right|}

\def\iff{\mathrel{\leftrightarrow}}

\def\and{\mathrel{\kern1pt\&\kern1pt}}

\def\elesub{\prec}

\def\inseg{\vartriangleleft}
\def\insegeq{\trianglelefteq}

\def\<#1>{\langle\,#1\,rnggle}

 \input xy
 \xyoption{all}

\title{The core model induction beyond $L(\mathbb{R})$: Non-tame mouse from PFA 
}

\author{Grigor Sargsyan\thanks{This material is partially based upon work supported by the National Science Foundation under Grant No DMS-0902628. Part of this paper was written while the author was a Leibniz Fellow at the Mathematisches Forschungsinstitut Oberwolfach.} \\
        Department of Mathematics\\
        Rutgers University, \\
        New Brunswick, NJ, 08854 USA\\
        http://math.rutgers.edu/$\sim$gs481\\
        grigor@math.rutgers.edu}

\date{\today}

\begin{document}

\maketitle

\begin{abstract}
Building on the work of Schimmerling (\cite{Sch}) and Steel (\cite{PFA}), we show that the failure of square principle at a singular strong limit cardinal implies that there is a non-tame mouse. This is the first step towards getting a model of $AD_{\mathbb{R}}+``\Theta$ is regular" from PFA via the \textit{core model induction}. 
\end{abstract}

\baselineskip=19pt

One of the wholly grails of inner model program has been determining the exact consistency strength of forcing axioms such as \textit{Proper Forcing Axiom} (PFA) and \textit{Martin's Maximum} (MM). As the consistency of PFA, MM and other similar axioms have been established relative to one supercompact cardinal, it is natural to conjecture that the exact consistency strength of such forcing axioms is one supercompact cardinal. 

\begin{conjecture}[The PFA Conjecture]\label{pfa conjecture} PFA is equiconsistent with one supercompact cardinal. 
\end{conjecture} 

Recently, in \cite{VW}, Viale and Weiss showed that any of the known methods for forcing PFA requires a strongly compact cardinal and if the forcing used is proper then it requires a supercompact cardinal. This result suggest that The PFA Conjecture must indeed be true.  

Many constructions in inner model theory have been motivated by The PFA Conjecture or by its sister conjectures. The \textit{core model induction}, a method for establishing reversals pioneered by Woodin and further developed by Schimmerling, Schindler, Steel and others, has been a very successful method for attacking  The PFA Conjecture. The known partial results, however, do not use the full force of PFA, but only a consequence of it, namely that $\square$ principle fails at some singular strong limit cardinal. Letting $Lim$ be the set of limit ordinals, recall that $\square_\kappa$ is the statement: there is a sequence $\la C_\a: \a\in Lim\cap \kappa^+\ra$ such that
\begin{enumerate}
\item $C_\a\subseteq \a$ is a club subset of $\a$,
\item if $\a\in C_\b\cap Lim$ then $C_\a=C_\b\cap \a$,
\item $o.t.(C_\a)\leq \k$.
\end{enumerate}

\begin{theorem}[Todorcevic, \cite{Jech}] PFA implies $\neg \square_\kappa$ for all $\k\geq \omega_1$. 
\end{theorem}

Building on an earlier work of Schimmerling and Woodin, Steel, in \cite{PFA}, established the following theorem.

\begin{theorem}[Steel, \cite{PFA}]\label{steel} Assume $\neg\square_\kappa$ for some singular strong limit cardinal $\k$. Then $AD$ holds in $L(\mathbb{R})$. In particular, there is an inner model with infinitely many Woodin cardinals. 
\end{theorem} 

The failure of the square principle is an anti inner model hypothesis as it usually implies $V$ isn't an inner model. For instance, Jensen showed that $L\models \forall \k\geq \omega_1\square_\k$. More generally, Schimmerling and Zeman showed that if $\M$ is a mouse then for $\k\geq \omega_1^\M$, $\M\models \square_\k$ iff $\M\models ``\kappa$ isn't subcompact" (see \cite{SchZem}). Subcompactness is a large cardinal notion reminiscent of supercompactness. It is much stronger than Woodin cardinals but it is weaker than supercompact cardinals. Unlike supercompactness, subcompact cardinals can exist in \textit{short extender mice}. The theory of such mice has been fully developed (see \cite{FSIT} or \cite{OIMT}). \rthm{SchZem} is the exact consequence of the aforementioned Schimmerling-Zeman result that we will need in this paper. 

Arranging the failure of $\square_\k$ for regular cardinal $\k$ is relatively easy. In fact, Velickovic showed that if $\kappa$ is a regular cardinal and $\l>\k$ is a Mahlo cardinal then after Levy collapsing $\l$ to $\k^+$, $\square_\k$ fails (see Theorem 1.3 of \cite{Sch}).  Arranging the failure of $\square_\k$ for a singular cardinal $\k$ is much harder and the known constructions use a subcompact cardinal. Reversing this result has been another wholly grail of the inner model program.

\begin{conjecture}
Suppose $\k$ is a singular cardinal such that $\neg\square_\k$ holds. Then there is an inner model with a subcompact cardinal. 
\end{conjecture}

The author recently generalized Steel's proof of \rthm{steel} and obtained a model containing the reals and ordinals and satisfying $AD_{\mathbb{R}}+``\Theta$ is regular". Recall that AD is the Axiom of Determinacy which states that all two player games of perfect information on integers are determined, i.e., one of the players has a winning strategy. $AD_{\mathbb{R}}$ is the axioms which states that all two player games of perfect information on reals are determined. We let
\begin{center}
$\Theta=\sup \{ \a :$ there is a surjection of $\mathbb{R}$ onto $\a\}$.
\end{center}
The axiom $AD_{\mathbb{R}}+``\Theta$ is regular" is a natural closure point of the \textit{Solovay hierarchy} which is a determinacy hierarchy whose base theory is $AD^+$. In the subsequent sections, we will give a more precise definitions of $AD^+$ and the Solovay hierarchy. 

\begin{theorem}\label{full theorem} Suppose $\neg \square_\k$ holds for some singular strong limit cardinal $\k$. Then there is $M$ such that $Ord, \bR\subseteq M$ and $M\models AD_{\bR}+``\Theta$ is regular".
\end{theorem}

\begin{corollary} Assume PFA. Then there is $M$ such that $Ord, \bR\subseteq M$ and $M\models AD_{\bR}+``\Theta$ is regular". 
\end{corollary}

In recent years, the Solovay hierarchy has been used as intermediary for establishing reversals. Using the core model induction, working under $PFA$ or any other theory in question, we construct a model satisfying some theory from the Solovay hierarchy. The next step, then, is to reduce the resulting theory from the Solovay hierarchy to one in large cardinal hierarchy. \rthm{full theorem} is an example of the first step of this process. Recently the author and Yizheng Zhu showed that $AD_{\mathbb{R}}+``\Theta$ is regular" is equiconsistent with $\Theta$-regular hypothesis, which is weaker than a Woodin limit of Woodin cardinals but much stronger than a proper class of Woodin cardinals and strong cardinals. For the details see \cite{BSL} or \cite{SargZhu}. We then get the following corollary.

\begin{corollary} Suppose PFA holds or just that $\neg \square_\k$ holds for some singular strong limit cardinal $\k$. Then there is an inner model satisfying ZFC+$\Theta$-regular hypothesis. In particular, there is an inner model with a proper class of Woodin cardinals and strong cardinals. 
\end{corollary}

The proof of \rthm{full theorem} uses the core model induction at the level of $AD_{\mathbb{R}}+``\Theta$ is regular" and basic machinery at this level hasn't yet appeared in print. Because of this, it is unrealistic to hope to present an intelligible proof of \rthm{full theorem}. Instead, we will concentrate on the first step of the induction and will present the full proof of \rthm{full theorem} in future publications. The main theorem of this paper is the following.

\begin{theorem}[Main Theorem]\label{main theorem} Suppose $\neg \square_\k$ holds for some singular strong limit cardinal $\k$. Then there is a non-tame mouse. 
\end{theorem}

\textbf{Remark.} The hypothesis of \rthm{main theorem} might seem somewhat confusing as we almost never use it to establish any of the facts in this paper. It is only used to establish \rthm{strategies are determined} and \rthm{determinacy in the max model}. If the conclusions of the aforementioned theorems were true, we could prove the conclusion of \rthm{main theorem} from just the assumption that various $Lp$-closures of $V_\k$ have height $<\k^+$. \rthm{summary} summarizes the exact hypothesis we need to carry out our proof of the conclusion of \rthm{main theorem}.

\section{Preliminaries}

\subsection{Premice, mice and their iteration strategies}

We establish some notation and list some basic facts about mice and their iteration strategies.
The reader can find more detail in \cite{OIMT}. Suppose $\M$ is a premouse. We  then let $o(\M)=Ord\cap \M$. If $\M$ is a premouse and $\xi\leq o(\M)$ then we let $\M||\xi$ be $\M$ cutoff at $\xi$, i.e., we keep the predicate indexed at $\xi$. We let $\M|\xi$ be $\M||\xi$ without the last predicate. We say $\xi$ is a \textit{cutpoint} of $\M$ if there is no extender $E$ on $\M$ such that $\xi\in(\cp(E), lh(E)]$. We say $\xi$ is a \textit{strong cutpoint} if there is no $E$ on $\M$ such that $\xi\in[\cp(E), lh(E)]$. We say $\eta<o(\M)$ is \textit{overlapped} in $\M$ if $\eta$ isn't a strong cutpoint of $\M$. Given $\eta<o(\M)$ we let
\begin{center}
$\mathcal{O}^M_\eta= \cup\{ \N\inseg\M: \rho(\N)=\eta $ and $\eta$ is not overlapped in $\N\}$. 
\end{center}

If $\M$ is a $k$-sound premouse, then a {\em $(k,\theta)$-iteration strategy} for $\M$
is a winning strategy for player $\rm{II}$ in the iteration game $G_k(\M,\theta)$,
and a {\em $k$-normal iteration tree} on $\M$ is a play of this game in which $\rm{II}$ has
not yet lost. (That is, all models are wellfounded.)
$k$-normal trees are called ``$k$-maximal" in 
\cite{OIMT}, but we shall use ``maximal" for a completely different property of trees here.)
We shall drop the reference to the fine-structural parameter $k$ whenever it
seems safe to do so, and speak simply of normal trees. 

We say $\M$ is $\theta$-iterable if $II$ has a winning strategy in  $G_k(\M,\theta)$. We say $\M$ is \textit{countably} $\theta$-iterable if any countable substructure of $\M$ is $\theta$-iterable. It follows from the copying construction that a $\theta$-iterable mouse is countably $\theta$-iterable. We say $\M$ is \textit{countably} iterable if all of it is countable substructures are $\omega_1+1$-iterable. 
 
If $\T$ is a normal iteration tree, then $\T$ has the form 
\begin{center} 
$\T=\la T, deg, D, \la E_\a, \M_{\a+1}| \a+1<\eta\ra\ra$. 
\end{center} 
Recall that $D$ is the set of \textit{dropping} points. Recall also that if $\eta$ is limit then 
\begin{center} 
$\vec{E}(\T)= \cup_{\a<\eta}(\dot{E}^{\M_\a\rest  lh(E_\a)})$,\\ 
$\M(\T)=\cup_{\a<\eta}\M_\a\rest lh(E_\a)$,\\ 
$\d(\T)=\sup_{\a<\eta} lh(E_\a)$. 
\end{center} 
If $b$ is a branch of $\T$ such that $D \cap b$ is finite,
 then $\M^\T_b$ is the direct limit  of the models along $b$ . If $\a \leq_T\b$
and $(\a, \b]_\T\cap D=\emptyset$ then  
\begin{center} 
 $\pi_{\a, \b}^\T:\M_\a^\T \rightarrow \M_\b^\T$ 
\end{center} 
 is the iteration map, and if $\a\in b$ and $(b-\a) \cap D =\emptyset$,
then  
\begin{center} 
 $\pi_{\a, b}^\T:\M^*_\a\rightarrow \M_b^\T$ 
\end{center} 
 is the iteration map. If $\T$ has a last model $\M_\a^\T$, and the branch
$[0,\alpha]_T$ does not drop, then we often write $\pi^\T$ for $\pi_{0,\a}^\T$.

 Recall that the strategy for a sound mouse projecting to $\omega$ is determined by
{\em $\Q$-structures}. For $\T$ normal, let $\Phi(\T)$ be the phalanx of $\T$. 

\begin{definition}\label{qstructures} Let $\T$ be a $k$-normal tree of limit
length on a $k$-sound
premouse, and let $b$ be a cofinal branch of $\T$. Then $\Q(b,\T)$ is the shortest
initial segment $\Q$ of $M_b^{\T}$, if one exists, such that
$\Q$ projects strictly across $\d(\T)$ or defines a function witnessing $\d(\T)$ is not 
Woodin via extenders on the sequence of $\M(\T)$.
\end{definition}

\begin{theorem}\label{strategyunique} Let $\M$ be a $k$-sound premouse such that
$\rho_k(\M) = \omega$. Then $\M$ has at most one $(k,\omega_1+1)$ iteration strategy.
Moreover, any such strategy $\Sigma$ is determined by: $\Sigma(\T)$ is the unique
cofinal $b$ such that the phalanx
 $\Phi(\T)^\frown \la \d(\T), \deg^{\T}(b), \Q(b,\T)\ra$ is $\omega_1+1$-iterable.
\end{theorem}  

In some cases, however, it is enough to assume that $\Q(b, \T)$ is countably iterable. This happens, for instance, when $\M$ has no local Woodin cardinals with extenders overlapping it. While the mice we will consider do have local overlapped Woodin cardinals, the mice themselves will not have such Woodin cardinals. This simplifies our situation somewhat and below we describe exactly how this will be used.  We say an iteration tree $\T$ is above $\eta$ if all the extenders used in $\T$ have critical points $>\eta$.

\begin{definition}\label{fatal drop} Suppose $\M$ is a premouse and $\T$ is a normal tree on $\M$. We say $\T$ has a \textit{fatal drop} if for some $\a$ such that $\a+1<lh(\T)$, there is some $\eta<o(\M_\a^\T)$ such that the rest of $\T$ is an iteration tree on $\mathcal{O}^{\M_\a^\T}_\eta$ above $\eta$. Suppose $\T$ has a fatal drop. Let $(\a, \eta)$ be lexicographically least witnessing the fact that $\T$ has a fatal drop. Then we say $\T$ has a fatal drop at $(\a, \eta)$.
\end{definition}

\begin{definition}[{Definition 2.1} of \cite{CMWMWC}]\label{qoft} Let $\T$ be a normal iteration tree;
then  $\Q(\T)$ is the unique premouse extending $\M(\T)$
 that has
$\d(\T)$ as a strong cutpoint, is $\omega_1+1$-iterable above $\d(\T)$, and either 
projects strictly across $\d(\T)$ or defines a function witnessing $\d(\T)$ is not
Woodin via extenders on the sequence of $\M(\T)$, if there is any such
premouse.
\end{definition}

Countable iterability is enough to guarantee there is at most one premouse with
the properties of $\Q(\T)$. If it exists, $\Q(\T)$ might identify the good branch of $\T$,
the one any sufficiently powerful iteration strategy must choose. This is the content of
the next lemma which can be proved by analyzing the proof of Theorem 6.12 of \cite{OIMT}

\begin{lemma}\label{qstructures} Suppose $\M$ is a $k$-sound
 premouse such that no measurable cardinal of $\M$ is a limit of Woodin cardinals. Suppose
   $\T$ is a $k$-normal iteration tree on $\M$ of limit length which doesn't have a fatal drop  and suppose $\Q(\T)$ exists.
 Then there is at most one cofinal branch $b$ of $\T$ such that
either $\Q(\T) = \M_b^{\T}$ or $\Q(\T) = \M_b^{\T} |\xi$ for some $\xi$ in the wellfounded
part of $\M_b^\T$.
\end{lemma} 

We shall need to look more closely at what is behind the uniqueness results above, at how the
``fragments" of $\Q(b,\T)$ (or better, its theory) determine the initial segments of
$b$. The following is the crucial lemma which is essentially due to Martin and Steel (\cite{IT})
 
 \begin{lemma}\label{uniqueness of branches} Suppose $\T$ is a normal iteration tree on $\M$ of
 limit length and $s$ is a cofinal subset of $\d(\T)$; then there is at most one cofinal branch $b$
 such that there is $\a\in b$ with the property that $i^\T_{\a, b}$ exists and $s\subseteq rng(i^{\T}_{\a, b})$. 
\end{lemma} 
\begin{proof} 
Towards a contradiction, suppose there are two cofinal branches $b$ and $c$ such that for some $\a, \b$, both $i^\T_{\a, b}$ and $i^\T_{\b, c}$ exist and $s\subseteq ran(i^\T_{\a, b})\cap ran(i^\T_{\b, c})$. Without loss of generality we can assume that $\a$ and $\b$ are the least ordinals with this property, $\a\leq \b$ and that $b$ and $c$ diverge at $\a$ or earlier, i.e., if $\gg$ is the least ordinal in $b\cap c$ then $\gg\leq \a$. By \cite{IT}, we can assume that $b=\la \a_n: n<\omega\ra$, $c=\la \b_n: n<\omega\ra$, $\a_0=\a$ and $\b_0=\b$. Let then $\xi$ be the least ordinal in $ran(i^\T_{\a, b})\cap ran(i^\T_{\b, c})$. Let $n$ be the least such that $\cp(i^\T_{\a_n, b})>\xi$. This means that $\cp(E^\T_{\a_{n+1}-1})>\xi$ and that $lh(E^\T_{\a_n})<\xi$. By the proof of Theorem 2.2 of \cite{IT}, this means that for some $m\geq 1$, $\xi \in [\cp(E^\T_{\b_m-1}), lh(E^\T_{\b_m-1}))$. This then implies that $\xi\not \in ran(i^\T_{\b_{m-1}, c})$, which is a contradiction. 
\end{proof} 
 
$\Q(b,\T)$ identifies $b$ because it determines a canonical cofinal subset
of $rng(i_{\a,b}^{\T} \cap \d(\T))$, for some $\a \in b$, to which we can apply \rlem{uniqueness of branches}.
But now, the proof of \rlem{uniqueness of branches} gives the following refinement: 
 
\begin{lemma}\label{partial agreement} Suppose $\T$ is an iteration tree on $\M$
 of limit length and $b, c$ are two cofinal branches of $\T$ such
 that $i^\T_b$ and $i^\T_c$ exist. Suppose that for some $\a$, 
\begin{center} 
$i^\T_b(\a)=i^\T_c(\a)<\d(\T)$. 
\end{center} 
Then $i^\T_b\rest \a = i^\T_c\rest \a$. Moreover, if $\xi\in b$ is the least such that
 $\cp(i_{\xi,b}^\T)>i^\T_b(\a)$ then $\xi \in c$, so that 
$b\cap (\xi +1) = c\cap (\xi +1)$. 
\end{lemma}

In addition to normal trees, we must consider linear stacks of normal trees. These are plays
of the iteration game $G_k(\M,\a,\theta)$ in which $\rm{II}$ has not yet lost. 
See \cite[Def. 4.4]{OIMT} for the formal definition. We shall generally use the vector notation
$\VT$ for a stack of normal trees, and then $\T_\eta$ for its $\eta^{\mbox{th}}$ normal component. $\M_\eta$ will denote the model at the beginning of the $\eta$th round of $\VT$. 
If $\Sigma$ is a $(k,\alpha,\theta)$-iteration strategy for $\M$, and $\VT$ is a stack of
trees on $\M$ played according to $\Sigma$ and having last model $\N$, then we let
$\Sigma_{\N,\VT}$ be the $(l,\alpha,\theta)$ strategy for $\N$ induced by $\Sigma$.
($l$ is the degree of the branch $\M$-to-$\N$ of $\VT$.
We assume here $\alpha$ is additively closed, so that there is such a strategy.) We say $\M$ is countable $(\a, \theta)$-iterable if all of its countable submodels are $(\a, \theta)$-iterable.

%
%
 
Suppose $\M$ is a mouse and $\Sigma$ is an $(\a, \theta)$-iteration strategy. We then let 
\begin{center} 
$I(\M, \Sigma)=\{ \N :$ there is a stack $\VT$ on $\M$ according to $\Sigma$ with last model $\N$ and $\pi^{\VT}$ exists $\}$. 
\end{center} 

Given a premouse $\M$ with a unique 
Woodin cardinal $\d$, we let $\mathbb{B}^{\M}$ be the countably generated extender algebra of $\M$ at
$\d$. In order to have a unique choice, we stipulate that the identities
determining  $\mathbb{B}^{\M}$ are precisely those coming from extenders $E$ on the
$\M$-sequence such that $\nu(E)$ is an inaccessible, but not a limit of inaccessibles,
in $\M$, and $\nu(E) < \d$.  If $G\subseteq \mathbb{B}^M$ then we let $x_G$ be the set naturally coded by $G$. 
For basic facts about the extender algebra, we refer the reader to
\cite{EA} and \cite{OIMT}.

\subsection{AD$^+$ and the Solovay hierarchy}

We will not need the exact formulation of $AD^+$. The interested readers can consult \cite{Woodin} or the introductory chapters of \cite{ATHM}. The \textit{Solovay hierarchy} is a hierarchy of axioms extending $AD^+$. To define the hierarchy, we first need to define the \textit{Solovay sequence}. First recall the \textit{Wadge ordering} of $\powerset(\mathbb{R})$. For $A, B\subseteq \mathbb{R}$, we say $A$ is \textit{Wadge reducible} to $B$ and write $A\leq_W B$ if there is a continuous function $f:\mathbb{R} \rightarrow \mathbb{R}$ such that $f^{-1}``B=A$. Martin showed that under $AD$, $\leq_W$ is a wellfounded relation. For $A\subseteq \mathbb{R}$, we let $w(A)$ be the rank of $A$ in $\leq_W$. Under $AD$, $\Theta=\sup_{A\subseteq \bR}w(A)$.

 The Solovay sequence is defined as follows. 
 \begin{definition}[The Solovay sequence] Assume $AD$. The Solovay sequence is a closed increasing sequence $\la \theta_\a : \a\leq \Omega\ra$ of ordinals defined by
\begin{enumerate}
\item $\theta_0=\sup\{ \a :$ there is an $OD$ surjection $f: \bR \rightarrow \a\}$,
\item if $\theta_{\b}<\Theta$ then letting $A\subseteq \bR$ be such that $w(A)=\theta_\b$,
    \begin{center}
    $\theta_{\b+1}=\sup \{ \a :$ there is an $OD_A$ surjection $f: \powerset(\theta_\b)\rightarrow \a\}$,
    \end{center}
\item if $\l$ is a limit then $\theta_\l=\sup_{\a<\l}\theta_\a$.
\end{enumerate}
\end{definition}

The Solovay hierarchy is the hierarchy we get by requiring that $\Omega$ is large. The following are the first few theories of this hierarchy. $\leq_{con}$ is the consistency strength relation: given two theories $T$ and $S$, $S\leq_{con} T$ if $Con(T)\vdash Con(S)$.
\begin{center}
$AD^{+}+\Omega=0<_{con} AD^{+}+\Omega=1 < AD^{+}+\Omega=2 \cdot \cdot \cdot <_{con} AD^{+}+\Omega=\omega <_{con}\cdot \cdot \cdot <_{con} AD^{+}+\Omega=\omega_1 <_{con} AD^{+}+\Omega= \omega_1+1 \cdot \cdot \cdot$.
\end{center}
For more on the Solovay hierarchy consult \cite{BSL}. We will use the following theorem to construct a non-tame mouse from our hypothesis. 

\begin{theorem}[Woodin, \cite{DMATM}]\label{equiconsistency} Assume there is a transitive inner model $M$ containing $\bR$ such that $M\models AD^++\Theta=\theta_1$. Then there is a non-tame mouse. 
\end{theorem}

\subsection{Suitable mice}\label{suitable mice}

The core model induction is a method for constructing iteration strategies for various models and often times the iteration strategies we construct are strategies for $\H$'s of various models of determinacy. In many situation, including our current situation, such $\H$'s are known to be fine structural models and the proof that they are will be used throughout this paper. Because of this we take a moment to review the background material for the $\H$ computation of models of determinacy. In later sections we will heavily rely on notions introduced in this section. In particular, notions such as \textit{suitable premouse} or \textit{quasi-iterable premouse} will be of crucial importance for us. Many of the notions introduced in this section are due to Woodin who introduced them in his computation of $\H^{L[x][g]}$ (see \cite{SSW}). We start with suitability.
 
 Often times, the notion of a suitable premouse is needed simultaneously in models of determinacy and in ZFC context. Our current situation is an instance of this and we chose to define the notion in a most general way to avoid confusions when applying it in different contexts. 

Fix then some cardinal $\lambda$ and let $\Gamma\subseteq \powerset(\powerset(\lambda))$. We assume $ZF+DC_\lambda$. Notice that any function $f:H_{\l^+}\rightarrow H_{\l^+}$ can be naturally coded by a subset of $\powerset(\lambda)$. We then let $Code_\lambda:  H_{\lambda^+}^{H_{\l^+}}\rightarrow \powerset(\powerset(\lambda))$ be one such coding. If $\l=\omega$ thence just write $Code$. Because any $\lambda^+$-iteration strategy for a premouse of size $\leq\l$ is in $H_{\lambda^+}^{H_{\l^+}}$, we have that any such strategy is in the domain of $Code_\l$. Given a premouse $\M$, we say $\M$ \textit{has an iteration strategy} in $\Gamma$ if $\card{\M}\leq \l$ and $\M$ has a $\lambda^+$-iteration strategy  (or $(\a, \lambda^+)$-iteration strategy for $\a\leq \l^+$) $\Sigma$ such that $Code_\l(\Sigma)\in\Gamma$. We let $Mice^\Gamma$ be the set of mice that have an iteration strategy in $\Gamma$.  Given a countable set $a$ we let 
\begin{center}
$\W^\Gamma(a)=\cup\{ \N: \N$ is a sound mouse over $a$ such that $\rho(\N)=a$ and $\N\in Mice^\Gamma\}$
\end{center}
and define $Lp^\Gamma_\xi(a)$ for $\xi<\omega_1$ by induction as follows:
\begin{enumerate}
\item $Lp^\Gamma_0(a)=\W^\Gamma(a)$,
\item for $\xi<\omega_1$, $Lp^\Gamma_{\xi+1}=\W^{\Gamma}(Lp^\Gamma_\xi(a))$,
\item for limit $\xi<\omega_1$, $Lp^\Gamma_{\xi}=\cup_{\a<\xi}Lp^\Gamma_\a(a)$.
\end{enumerate}

\begin{definition}[$\Gamma$-suitable premouse] \label{suitable premouse}
 A premouse $\P$ is $\Gamma$-suitable if there is a unique cardinal $\delta$ such that
\begin{enumerate}
 \item $\P\models ``\d$ is the unique Woodin cardinal",
 \item $o(\P)=\sup_{n<\omega} (\d^{+n})^\P$,
 \item for every $\eta\not =\d$, $\W^\Gamma(\P|\eta)\models ``\eta$ isn't Woodin".
 \item for any $\eta<o(\P)$, $\mathcal{O}_\eta^\P=\W^{\Gamma}(\P|\eta)$.
 \end{enumerate}
\end{definition}

Suppose $\P$ is $\Gamma$-suitable.  Then we let $\d^\P$ be the $\d$ of \rdef{suitable premouse}. Given an iteration tree $\T$ on $\P$, we say $\T$ is \textit{nice} if $\T$ has no fatal drops. Notice that $\Gamma$-suitable premice satisfy the hypothesis of \rlem{qstructures}. A nice tree $\T$ is 
\textit{$\Gamma$-correctly guided} if for every limit $\a<lh(\T)$, $\Q(\T\rest \a)$ exists and 
\begin{center}
 $\Q(\T\rest \a)\trianglelefteq \W^{\Gamma}(\M(\T\rest \a))$.
\end{center}
 $\T$ is $\Gamma$-\textit{short} if it is nice, correctly guided and $\W^\Gamma(\M(\T))\models ``\d(\T)$ is not Woodin". $\T$ is $\Gamma$-\textit{maximal} if it is nice, $\Gamma$-correctly guided yet not $\Gamma$-short. Notice that if $\T$ is a maximal tree and $b$ is a branch such that $i_b^\T(\d^\P)=\d(\T)$ then $\T$ doesn't have a nice normal continuation. 

\begin{definition}[$\Gamma$-correctly guided finite stack]
 Suppose $\P$ is $\Gamma$-suitable. We say $\la \T_i, \P_i : i<m\ra$ is a $\Gamma$-\textit{correctly guided finite
stack} on $\P$ if
\begin{enumerate}
\item $\P_0=\P$,
\item $\P_i$ is $\Gamma$-suitable and $\T_i$ is a nice $\Gamma$-correctly guided tree on $\P_i$ below $\d^{\P_i}$,
\item for every $i$ such that $i+1< m$ either $\T_i$ has a last model and $\pi^\T$-exists or $\T$ is
maximal, and
    \begin{enumerate}
    \item if $\T_i$ has a last model then $\P_{i+1}$ is the last model of $\T_i$ and if $lh(\T_i)=\a+1$ where $\a$ is limit then if $\T^-_i$ is $\T_i$ without its last branch then $\Q(\T^-_i)$-exists,
    \item if $\T_i$ is $\Gamma$-maximal then $\P_{i+1}=Lp_\omega^{\Gamma}(\M(\T_i))$.
    \end{enumerate}
\end{enumerate}
\end{definition}

Notice that if $\la \T_i, \P_i : i\leq m\ra$ is a correctly guided finite  stack on $\P$ then only
$\T_m$ can have a dropping last branch. 

\begin{definition}[The last model of a $\Gamma$-correctly guided finite stack]\label{last model of finite stack}
 Suppose $\P$ is $\Gamma$-suitable and $\VT=\la \T_i, \P_i : i\leq k\ra$ is a $\Gamma$-correctly guided finite stack
on $\P$. We say $\R$ is the last model of $\VT$ if one of the following holds:
 \begin{enumerate}
 \item $\T_{k}$ has a last model and $\R$ is the last model of $\T_{k}$,
 \item $\T_k$ is of limit length, $\T_{k}$ is $\Gamma$-short and there is a cofinal well-founded branch $b$
such that $\Q(b, \T_k)$ exists, $\Q(b, \T)\insegeq \W^\Gamma(\M(\T_k))$ and $\R=\M^\T_b$,
 \item $\T_k$ is of limit length, $\T_{k}$ is $\Gamma$-maximal, $\R$ is $\Gamma$-suitable and
 \begin{center}
$\R=Lp_{\omega}^\Gamma(\M(\T_{k}))$.
\end{center}
 \end{enumerate}
We say $\R$ is a \textit{correct iterate} of $\P$ if there is a correctly guided finite stack on $\P$ with last model
$\R$.
\end{definition}

\begin{definition}[$S(\Gamma)$ and $F(\Gamma)$] We let $S(\Gamma)=\{ \Q: \Q$ is $\Gamma$-suitable$\}$. Also, we let $F(\Gamma)$ be the set of functions $f$ such that $dom(f)=S(\Gamma)$ and for each $\P\in S(\Gamma)$, $f(\P)\subseteq \P$ and $f(\P)$ is amenable to $\P$, i.e., for every $X\in \P$, $X\cap f(\P)\in \P$.
\end{definition}

 Given $\P\in S(\Gamma)$ and $f\in F(\Gamma)$ we let $f^n(\P)=f(\P)\cap \P|((\d^\P)^{+n})^\P$. Then $f(\P)=\cup_{n<\omega}f^n(\P)$. We also let 
\begin{center}
$\gg^\P_{f}=\d^\P \cap Hull^{\P}_1( \{ f^n(\P) : n<\omega \}) $.
\end{center}
Notice that
\begin{center}
$\gg^\P_{f}=\d^\P \cap Hull^{\P}_1(\gg_{f}^\P\cup\{ f^n(\P) : n<\omega \} )$.
\end{center}
We then let 
\begin{center}
$H_{f}^\P =Hull^{\P}_1(\gg^\P_{f}\cup \{ f^n(\P) : n<\omega \} )$.
\end{center}
If $\P\in S(\Gamma)$, $f\in F(\Gamma)$ and $i: \P\rightarrow \Q$ is an embedding then we let $i(f(\P))=\cup_{n<\omega}i(f^n(\P))$.

\begin{definition}[$f$-iterability]\label{f-iterability} Suppose $\P\in S(\Gamma)$ and $f\in F(\Gamma)$. We say $\P$ is $f$-iterable if whenever $\la \T_k, \P_k : k<m\ra$ is a
finite correctly guided stack on $\P$ with last model $\R$ then there is a sequence $\la b_k: k<
m\ra$ such that the following holds.
\begin{enumerate}
\item For $k<m-1$,
 \begin{center}
   $b_k = \begin{cases}
           \emptyset  &:  \T_k \ \text{has a successor  length} \\
      \text{cofinal well-founded branch}\\ \text{such that}\ \M^\T_{b_k}=\P_\k &:  \T_k \ \text{is
$\Gamma$-maximal}
          \end{cases}$
\end{center}
\item The following three cases hold.
\begin{enumerate}
 \item If $\T_{m-1}$ has a successor length then $b_{m-1}=\emptyset$.
 \item If $\T_{m-1}$ is $\Gamma$-short then there is a cofinal well-founded branch $b$ 
   such that $\Q(b, \T_{m-1})$ exists, $\Q(b, \T_{m-1})\insegeq\W^\Gamma(\M(\T_{m-1}))$ and $b_{m-1}$ is the unique such branch.
 \item  If $\T_{m-1}$ is $\Gamma$-maximal then $b_{m-1}$ is a cofinal well-founded branch. 
\end{enumerate}

\item Letting
 \begin{center}
   $\pi_k = 
     \begin{cases}
      \pi^{\T_k} &:  \T_k \ \text{has a successor length} \\
      \pi_{b_k}^{\T_k} &: \T_k \ \text{is $\Gamma$-maximal}
     \end{cases}$
   
\end{center}
and $\pi=\pi_{m-1}\circ \pi_{m-2}\circ\cdot \cdot\cdot \pi_0$ then 
    \begin{center}
  $\pi(f(\P))=f(\R)$.
    \end{center}
\end{enumerate}
\end{definition}

Suppose again that $\P\in S(\Gamma)$ and $f\in F(\Gamma)$. Suppose $\VT=\la \T_k, \P_k : k<m\ra$ is a $\Gamma$-correctly guided finite stack on $\P$ with last model $\R$. We
say \textit{$\vec{b}=\la b_k: k< m\ra$ witness $f$-iterability for $\VT=\la \T_k, \P_k : k< m\ra$}
if 2 above is satisfied. We then let
 \begin{center}
   $\pi_{\VT, \vec{b}, k} = 
     \begin{cases}
      \pi^{\T_k} &: \T_k \ \text{has a successor length} \\
      \pi_{b_k}^{\T_k} &: \T_k \ \text{is maximal}\\
     \text{undefined} &: \text{otherwise}
     \end{cases}$
\end{center}
and $\pi_{\VT, \vec{b}}=\pi_{\VT, \vec{b}, m-1}\circ \pi_{\VT, \vec{b}, m-2}\circ\cdot \cdot\cdot
\pi_{\VT, \vec{b}, 0}$. Notice that clause three isn't vacuous as it might be that $\T_k$ is $\Gamma$-short
and its unique branch has a drop.  

Continuing with the notation of the previous paragraph, let  $\vec{b}$ and $\vec{c}$ be two $f$-iterability branches for $\VT$. It then follows
from \rthm{uniqueness of branches} that 
\begin{center}
 $\pi_{\VT, \vec{b}}\rest H_{f}^\P=\pi_{\VT, \vec{c}}\rest H_f^\P$. 
\end{center}
\begin{lemma}[Uniqueness of $f$-iterability embeddings] \label{uniqueness of f-iterability
embeddings} Suppose $\P\in S(\Gamma)$, $f\in F(\Gamma)$ and $\VT$ is a
finite correctly guided stack on $\P$. Suppose $\vec{b}$ and $\vec{c}$ are two $f$-iterability
branches for $\VT$. Then
\begin{center}
$\pi_{\VT, \vec{b}}\rest H_{f}^\P=\pi_{\VT, \vec{c}}\rest H_f^\P$.
\end{center}
Moreover, if $\VT$ consists of just one normal tree $\T$, $\Q$ is the last model of $\T$ and $b$ and
$c$ witness $f$-iterability for $\T$ then if $\xi\in b$ is the least such that
$\cp(E^\T_\xi)>\gg_f^\Q$ then $b\cap \xi =c\cap \xi$. 
\end{lemma}

\begin{definition}\label{bf} Suppose $\P\in S(\Gamma)$ and $f$-iterable. Given a $\Gamma$-correctly guided maximal $\T$ on $\P$ with last model $\Q$, we let $b_{\T, f}=b\cap \xi$ where $b$ witnesses $f$-iterability of $\P$ for $\T$ and $\xi\in b$ is the least such that
$\cp(E^\T_\xi)>\gg_f^\Q$.
\end{definition}

Notice that, if $\P$ is $f$-iterable, $\VT$ is a correctly guided finite stack on $\P$, and
$\vec{b}$ witnesses $f$-iterability of $\P$ for $\VT$, then even though $\pi_{\VT, \vec{b}}\rest
H_{f}^\P$ is independent of $\vec{b}$ it may very well depend on $\VT$. This observation motivates
the following definition.

\begin{definition}[Strong $f$-iterability]\label{strong f-itearbility} Suppose $\P\in S(\Gamma)$ and $f\in F(\Gamma)$. We say
$\P$ is \emph{strongly $f$-iterable} if $\P$ is $f$-iterable  and whenever
$\VT=\la \T_j , \P_j : j< u\ra\in H_{\l^+}$ and $\VU=\la \U_j, \P_j : j < v\ra\in H_{\l^+}$ are two
correctly guided finite stacks on $\P$ with common last model $\R$, $\vec{b}$ witnesses
$f$-iterability for $\VT$ and $\vec{c}$ witnesses $f$-iterability for $\VU$ then $\pi^{\VT}_b$ is
defined iff $\pi^{\VU}_c$ is defined and
\begin{center}
$\pi_{\VT, \vec{b}}\rest H_f^\P = \pi_{\VU, \vec{c}}\rest H_f^\P$.
\end{center}
\end{definition}

If $\P$ is strongly $f$-iterable and $\VT$ is a correctly guided finite stack on $\P$ with last
model $\R$ then we let
\begin{center}
$\pi_{\P, \R, f}:H_f^\P\rightarrow H_f^\R$
\end{center}
be the embedding given by any $\vec{b}$ which witnesses the $f$-iterability of $\VT$, i.e., fixing
$\vec{b}$ which witnesses $f$-iterability for $\VT$,
\begin{center}
$\pi_{\P, \R, f} =\pi_{\VT, \vec{b}}\rest H_f^\P$.
\end{center}
Clearly, $\pi_{\P, \R, f}$ is independent of $\VT$ and $\vec{b}$. 

Given a finite sequence of functions $\vec{f}=\la f_i : i<n\ra \in F(\Gamma)$, we let $\oplus_{i<n}f_i\in F(\Gamma)$ be the function given by $(\oplus_{i<n}f_i)(\P)=\la f_i(\P): i<n\ra$. We set $\oplus\vec{f}= \oplus_{i<n}f_i$. 

We then let
\begin{center}
$\mathcal{I}_{\Gamma, F}=\{ (\P, \vec{f}): \P\in S(\Gamma)$, $\vec{f}\in (F(\Gamma))^{<\omega}$ and $\P$ is strongly $\oplus\vec{f}$-iterable$\}$.
\end{center}

\begin{definition}
Given $F\subseteq F(\Gamma)$, we say $F$ is \textit{closed} if for any $\vec{f}\subseteq  F^{<\omega}$ there is $\P$ such that $(\P, \oplus\vec{f})\in \mathcal{I}_{\Gamma, F}$ and for any $\vec{g}\subseteq F^{<\omega}$, there is a $\Gamma$-correct iterate $\Q$ of $\P$ such that $(\Q, \vec{f}\cup\vec{g})\in \mathcal{I}_{\Gamma, F}$.
\end{definition}  

Fix now a closed $F\subseteq F(\Gamma)$. Let 
\begin{center}
$\mathcal{F}_{\Gamma, F}=\{ H^\P_f: (\P, f)\in \mathcal{I}_{\Gamma, F}\}$.
\end{center}
We then define $\preceq_{\Gamma, F}$ on $\mathcal{I}_{\Gamma, F}$ by letting $(\P, \vec{f})\preceq_{\Gamma, F} (\Q, \vec{g})$ iff $\Q$ is a $\Gamma$-correct iterate of $\P$ and $\vec{f}\subseteq \vec{g}$. Given $(\P, \vec{f})\preceq_{\Gamma, F} (\Q, \vec{g})$, we have that 
\begin{center}
$\pi_{\P, \Q, \vec{f}}: H^\P_{\oplus\vec{f}}\rightarrow H^\Q_{\oplus\vec{f}}$.
\end{center}
Notice that if $F$ is closed then $\preceq_{\Gamma, F}$ is directed. Let then 
\begin{center}
$\M_{\infty, \Gamma, F}$
\end{center}
be the direct limit of $(\mathcal{F}_{\Gamma, F}, \preceq_{\Gamma, F})$ under $\pi_{\P, \Q, \vec{f}}$'s. Given $(\P, \vec{f})\in \mathcal{I}_{\Gamma, F}$, we let $\pi_{\P, \vec{f}, \infty}: H^\P_{\oplus\vec{f}}\rightarrow \M_{\infty, \Gamma, F}$ be the direct limit embedding. 

\begin{lemma} $\M_{\infty, \Gamma, F}$ is wellfounded.
\end{lemma}
\begin{proof}
If not then we can fix $\la (\P_i, f_i): i<\omega\ra\subseteq F$ and $\la \a_i: i<\omega\ra$ such that $\a_i\in H^{\P_i}_{f_i}$, $(\P_i, \oplus_{j\leq i} f_j)\in \mathcal{I}_{\Gamma, F}$ and $\pi_{\P_{i+1}, f_{i+1}, \infty}(\a_{i+1})<\pi_{\P_{i}, f_{i}, \infty}(\a_{i})$. By simultaneously comparing $\P_i$'s we get a common $\Gamma$-correct iterate $\P$ such that if $\b_i=\pi_{\P_i, \P, f_i}(\a_i)$ then $\b_{i+1}<\b_i$, contradiction!
\end{proof}

It turns out that $V_\Theta^\H$ of many models of determinacy can be obtained as $\M_{\infty, \Gamma, F}$ for some $\Gamma$ and $F$. We will give the details in the next few subsection. 

\subsection{Quasi-iterability}\label{quasi iterability section}

We are now in a position to define quasi-iterability which will later be used to construct strategies for certain $\H$s of models of determinacy. We continue with the set up of \rsec{suitable mice}. Thus, recall that we are assuming $ZF$ and $\Gamma\subseteq \powerset(\powerset(\lambda))$ for some cardinal $\l$. Fix some $\Gamma$-suitable $\P$ and a closed $F\subseteq F(\Gamma)$. Let $G\subseteq F$.

\begin{definition}\label{semi quasi iteration} We say $\la \T_\a, \M_\a: \a<\nu<\l^+\ra$ is a semi $(F, G)$-quasi iteration of $\P$ of length $\nu$ if 
\begin{enumerate}
\item $\M_0=\P$ and for all $\a<\nu$, $\M_\a$ is $\Gamma$-suitable,
\item for $\a<\nu$ then $\T_\a$ is a $\Gamma$-correctly guided tree on $\M_\a$ and if $\a+1<\nu$ then $\M_{\a+1}$ is the last model of $\T_\a$,
\item for all $\a<\nu$, $\M_\a$ is strongly $f$-iterable for every $f\in G$,
\item for all limit $\a<\nu$, for any $g\in F$ there is $\b<\a$ such that for any $\gg\in [\b, \a)$ 
\begin{enumerate}
\item if $lh(\T_\gg)$ is a successor then $i^{\T_\gg}(g(\M_\gg))=g(\M_{\gg+1})$,
\item if $lh(\T_\gg)$ is limit then there is a cofinal wellfounded branch $b$ of $\T_\gg$ such that $\M^{\T_\gg}_b=\M_{\gg+1}$ and $i^{\T_\gg}_b(g(\M_\gg))=g(\M_{\gg+1})$, 
\end{enumerate}
\end{enumerate}
\end{definition}

Suppose $\la \T_\a, \M_\a: \a<\nu<\l^+\ra$ is a semi $(F, G)$-quasi iteration of $\P$ of length $\nu$. Given limit $\a<\nu$ and $g\in F$ we let $\b_{g, \a}<\a$ be the least ordinal satisfying 4b of \rdef{semi quasi iteration}. 

\begin{definition}\label{embedding of semi quasi iteration} Suppose $\la \T_\a, \M_\a: \a<\nu<\l^+\ra$ is a semi $(F, G)$-quasi iteration of $\P$ of length $\nu$. We say $\la \pi^{g, \a}_{\gg, \xi} : \gg<\xi <\a\leq \nu \wedge \a\in Lim\cap \nu+1\ra$ are the embeddings of $\la \T_\a, \M_\a: \a<\nu<\l^+\ra$ if for all limit $\a\leq \nu$ and for all $g\in F$
\begin{enumerate}
\item $\pi^{g, \a}_{\gg, \xi}$ is defined for all $\gg, \xi\in [\b_{g,\a} , \a]$,
\item $\pi^{g, \a}_{\gg, \gg+1}=i^{\T_\gg}$ if it exists and otherwise $\pi^{g, \a}_{\gg, \gg+1}=i^{\T_\gg}_b\rest H_g^{\M_\gg}$ where $b$ witnesses clause 4b for $\a$ and $g$, 
\item fixing $\gg\in [\b_{g, \a}, \a]$, $\pi^{g, \a}_{\gg, \xi}$ is defined by induction on $\xi\in (\gg,\a]$ where the first step of the induction is clause 2 above and the rest is given by the following scheme:
\begin{enumerate}
\item if $\xi\leq \a$ is limit then if $H_g^{\M_\xi}$ is a direct limit of $H_\xi^{\M_\l}$'s under $\pi_{\gg, \l}^{g, \a}$'s for $\gg<\l<\xi$ then $\pi_{\gg, \xi}^{g, \a}$ is the direct limit embedding and otherwise it is undefined,
\item if $\xi=\l+1$ then $\pi_{\gg, \xi}^{g, \a}=\pi_{\l, \l+1}^{g, \a}\circ \pi_{\gg, \l}^{g, \a}$.
\end{enumerate} 
\end{enumerate} 
\end{definition}

Next we define $(F, G)$-quasi iterations.

\begin{definition}\label{quasi iteration} Suppose $\la \T_\a, \M_\a: \a<\nu<\l^+\ra$ is a semi $(F, G)$-quasi iteration of $\P$ of length $\nu$. We say $\la \T_\a, \M_\a: \a<\nu<\l^+\ra$ is a $(F, G)$-quasi iteration of $\P$ of length $\nu$ if letting $\la \pi^{g, \a}_{\gg, \xi} : \gg<\xi <\a\leq \nu \wedge \a\in Lim\cap \nu+1\ra$ be the embeddings of $\la \T_\a, \M_\a: \a<\nu<\l^+\ra$  then
\begin{enumerate}
\item  for every $\a$, $\M_\a=\cup_{g\in F} H^{\M_\a}_g$, and
\item for every limit $\a< \nu$ and for every $g\in F$, $H^{\M_\a}_g$ is the direct of $H^{\M_\gg}_g$ for $\gg\in[b_{g,\a}, \a)$ under the maps $\pi^{g, \a}_{\gg, \xi}$'s for $\gg<\xi\in [\b_{g, \a}, \a)$.
\end{enumerate}
\end{definition}

Finally we define the last model of $(F, G)$-quasi iterations.

\begin{definition}\label{last model of quay iteration} Suppose $\la \T_\a, \M_\a: \a<\nu<\l^+\ra$ is a $(F, G)$-quasi iteration of $\P$ of length $\nu$ with embeddings $\la \pi^{g, \a}_{\gg, \xi} : \gg<\xi <\a\leq \nu \wedge \a\in Lim\cap \nu+1\ra$ . Then we say $\Q$ is the last model of $\la \T_\a, \M_\a: \a<\nu<\l^+\ra$ if
$\Q$ is $\Gamma$-suitable, $\Q$ is strongly $f$-iterable for all $f\in G$ and one of the following holds:
\begin{enumerate}
\item  $\nu=\a+1$, $\Q$ is the last model of $\T_\a$ as defined in \rdef{last model of finite stack} and $\Q=\cup_{f\in F}H^\Q_f$ or
\item $\nu$ is limit, $\Q=\cup_{g\in F}H^\Q_g$ and for $g\in F$, $H^\Q_g$ is the direct limit of $H^{\M_\gg}_g$ under $\pi^{g, \nu}_{\gg, \xi}$'s for $\gg<\xi\in [\b_{g, \nu}, \nu)$.
\end{enumerate}
\end{definition}

Notice that it may be the case that quasi iterations don't have last models but when they do it is uniquely determined. 

\begin{definition}\label{quasi iterable} We say $\P$ is $(F, G)$-quasi iterable if all of its $(F, G)$-quasi iterations have a last model. 
\end{definition}

Before moving on we fix some notation. When $G=\{ f\}$ then we use $(F, f)$ instead of $(F, G)$. When $G=\emptyset$ then we write $F$ instead of $(F, \emptyset)$. Suppose $F\subseteq F(\Gamma)$ is closed. We say $\P$ is (strongly) $F$-iterable if $\P$ is (strongly) $f$-iterable for all $f\in F$. Suppose now that $\P$ is (strongly) $F$-iterable. Let $\Q$ be a correctly guided iterate of $\P$. Then we let $\pi_{\P, \Q, F}=\cup_{f\in F}\pi_{\P, \Q, f}$. Suppose further there is some $F^*\subseteq F(\Gamma)$ such that $\P$ is $(F^*, F)$-quasi iterable and suppose that $\Q$ is a $(F^*, F)$-quasi iterate of $\P$. We then let $\pi_{\P, \Q, F}:\cup_{f\in F}H_f^\P\rightarrow \cup_{f\in F}H_f^\Q$ be the iteration embedding coming from the composition of the quasi iterability embeddings. We also define $\pi_{\Q, F, \infty}$ similarly.

\subsection{$\vec{f}$-guided strategies}\label{f-guided strategies}

We continue with the set up of \rsec{suitable mice}. Thus, recall that we are assuming $ZF+DC_\l$.  In this subsection, our goal is to develop tools for constructing strategies by putting together pieces of various $f$-iterability branches. In later sections, we will need to construct iteration strategies that have the so-called branch condensation and one method for producing such strategies is via producing strongly $\vec{f}$-guided strategies. Below we let $ZFC^-$ stand for $ZFC$ without Replacement or without Powerset Axiom. We start by introducing \textit{branch condensation}, a notion that played an important role in \cite{ATHM}.

\begin{definition}\label{branch condensation} Suppose $\Sigma$ is an $(\a, \b)$-iteration strategy for some structure $M$ such that $M\models ZFC^-$ ($M$ need not be a fine structural model).  We say $\Sigma$ has branch condensation if whenever $(\R, \VT, \Q,  \VU, c, \pi)$ is such that 
\begin{enumerate}
\item $\VT$ is a stack on $M$ according to $\Sigma$ with last model $\R$ and $i^\VT$ exists,
\item $\VU$ is a stack on $M$ according to $\Sigma$ such that its last normal component is of limit length and $c$ is some branch on $\VU$ such that $\pi^{\VU}_c$ exists,
\item $\Q=\M^{\VU}_c$ and $\pi:\Q\rightarrow_{\Sigma_1}\R$ is such that $i^{\VT}=\pi\circ i^{\VU}_c$ 
\end{enumerate} 
then $c=\Sigma(\VU)$.
\end{definition}

Below we introduce one of the main methods for producing strategies with branch condensation. The basic idea is that if the strategy moves many operators correctly then it must have branch condensation. For example, let $\M=\M_1|(\d^+)^{\M_1}$ where $\M_1$ is the minimal mouse with a Woodin cardinal and $\d$ is the Woodin cardinal of $\M_1$. Then $\M$ may have many iteration strategies but it has one strategy that moves the theory of sharps correctly. Let $\Sigma$ be this strategy. It can be shown, using the fact that the theory of sharps \textit{condenses}, that $\Sigma$ has branch condensation. Below we will explain how exactly the aforementioned condensation of the theory of sharps works in a more general context. 

\begin{definition} Suppose $\P\in S(\Gamma)$ and $\Sigma$ is a $(\lambda^+, \lambda^+)$-iteration strategy for $\P$. We say $\Sigma$ is $\Gamma$-fullness preserving if whenever $i:\P\rightarrow \Q$ comes from an iteration produced by $\Sigma$, $\Q\in S(\Gamma)$. Given $f\in F(\Gamma)$, we say $\Sigma$ respects $f$ if whenever $i:\P\rightarrow \Q$ and $j:\Q\rightarrow \R$ are iterations produced via $\Sigma$ then $j(f(\Q))=f(\R)$. We say $\Sigma$ strongly respects $f$ if whenever $i, j, \P, \Q, \R$ are as above and $\S$ is such that there are $\sigma:\Q\rightarrow \S$ and $\tau:\S\rightarrow \R$ such that $j=\tau\circ \sigma$ then $\S\in S(\Gamma)$ and $\sigma(f(\Q))=f(\S)$.
\end{definition}

If $\Sigma$ strongly respects many $f$'s then it is possible to show that $\Sigma$ has branch condensation. Fix some $F\subseteq F(\Gamma)$.

\begin{definition}\label{qsjs}
We say $F$ is a quasi-self-justifying-system (qsjs) if there is $\P\in S(\Gamma)$ such that
\begin{enumerate}
\item $F$ is closed,
\item for any $\vec{f}\in F^{<\omega}$, $\oplus\vec{f}\in F$,
\item for every $f\in F$, $\P$ is $(F, f)$-quasi iterable,
\item whenever $\Q$ is a $(F, f)$-quasi iterate of $\P$, $\sup_{f\in F}\gg^\Q_f=\Q$ (this condition follows from clause 3 above),
\item whenever $\Q$ is a $F$-quasi iterate of $\P$ and $\sigma:\R\rightarrow_{\Sigma_1} \Q$ is such that $f(\Q)\in rng(\sigma)$ for all $f\in F$ then $\R\in S(\Gamma)$.
\end{enumerate}   
\end{definition}

The following lemmas are essentially due to Woodin and are based on \rlem{uniqueness of branches}. They collectively imply that if $F$ is a qsjs then it produces a $\Gamma$-fullness preserving iteration strategy for some $\P\in S(\Gamma)$.

\begin{lemma}[ZF]\label{one step lemma} Suppose $F$ is a qsjs and $\P\in S(\Gamma)$ witnesses it. Then there is a  $\lambda^+$-iteration strategy $\Sigma^F$ for $\P$ such that $\Lambda^F$ is $\Gamma$-fullness preserving.
\end{lemma}
\begin{proof} We define $\Lambda^F$ for normal trees. First we describe $\Lambda^F$ on $\Gamma$-correctly guided trees. Given such a tree $\T$ on $\P$ we have two cases. If $\T$ is $\Gamma$-short then we let $\Lambda^F(\T)=b$ be the unique branch of $\T$ such that $\Q(b, \T)$-exists. If $\T$ is a maximal then we do the following. Let $\Q$ be the last model of $\T$. Recall the definition of $b_f$ for $f\in F$ (see \rdef{bf}). For each $f\in F$, let $b_f$ be a branch of $\T$ witnessing $f$-iterability of $\P$ for $\T$. Because $\sup_{f\in F}\gg_f^\Q=\d^\Q$, we have that $b=\cup_{f\in F}b_f$ is a cofinal branch of $F$. Notice now that if $\S=\cup_{f\in F}\pi_b^{\T}(H_f^\P)$ then there is $\sigma:\S\rightarrow_{\Sigma_1} \Q$. By clause 5 of \rdef{qsjs}, we have that $\S\in S(\Gamma)$ and by clause 4 of \rdef{qsjs}, we have that $\sigma\rest \d(\T)=id$. Hence, $\S=\Q$. We then let $\Lambda^F(\T)=b$.

If $\T$ has a fatal drop then let $(\a, \eta)$ be such that $\T$ has a fatal drop at $(\a, \eta)$. Then we let $\Lambda^F(\T)$ be the $\Lambda(\T^*)$ where $\Lambda$ is the strategy of $\mathcal{O}_\eta^M$ and $\T^*$ is the $\T$ from $\a$ on. 
\end{proof}

Notice that if $F$ is a qsjs and $\P\in S(\Gamma)$ witnesses it then any $F$-quasi iterate of $\P$ also witnesses that $F$ is a qsjs. Given then any $\Q\in S(\Gamma)$ witnessing that $F$ is a qsjs we let $\Lambda^F_\Q$ be the $\l^+$-iteration strategy given by \rlem{one step lemma}. Let then $\P\in S(\Gamma)$  witness that $F$ is qsjs. To get a $(\l^+, \l^+)$-strategy for $\P$, we just need to dovetail all of $\Lambda^F_\Q$'s together for $\Q$'s that are $F$-quasi iterate of $\P$. The following notion will be used in the proof of the next lemma. 

Given an iteration tree $\T$ such that $lh(\T)=\gg+1$ where $\gg$ is limit, we let $\T^-$ be $\T$ without its last branch.
Suppose $\VT$ is a stack on a $\Gamma$-suitable premouse $\P$. We say $\VT$ is \textit{good} if letting $\la \T_\a, \M_\a: \a<\eta\ra$ be the normal components of $\VT$, for each $\a<\eta$, $\M_\a$ is $\Gamma$-suitable and $\T_\a$ is a $\Gamma$-correctly guided tree on $\M_\a$. We let $\VT^*$ be defined as follows: $\VT^*=\la \U_\a, \M_\a: \a<\eta\ra$ where $\U_\a=\T_\a$ if either $\T_\a^-$ is undefined or it is defined but it is $\Gamma$-short and otherwise, $\U_\a=\T_\a^-$. We say $\VT^*$ is the quasi-rearrangement of $\VT$.

\begin{lemma}[ZF]\label{getting strategies from qsjs} Suppose $F$ is a qsjs and $\P\in S(\Gamma)$ witnesses it. Then there is a $(\lambda^+, \lambda^+)$-iteration strategy $\Sigma^F$ for $\P$ such that $\Sigma^F$ is $\Gamma$-fullness preserving. 
\end{lemma}
\begin{proof} We handle trees with fatal drops exactly the same way as in \rlem{one step lemma}. We leave the details to the reader. By induction we define $\la \Sigma^F_\a: \a<\l^+\ra$ such that 
\begin{enumerate}
\item $\Sigma_\a^F$ is a $\Gamma$-fullness preserving $(\a, \l^+)$-iteration strategy for $\P$,
\item for $\a<\b\leq\l^+$, $\Sigma_\a^F\subseteq \Sigma_\b^F$, 
\item for all $\a<\l^+$, whenever $\VT$ is according to $\Sigma^F_\a$, $\VT^*$ is an $F$-quasi iteration of $\P$, and
\item if $\b$ is limit and $\VT=\la \T_\a, \M_\a:\a<\b\ra$ is a stack on $\P$ such that for each $\a<\b$, $\oplus_{\xi<\a}\T_\a$ is according to $\Sigma^F_\a$ then letting $\M$ be the direct limit along the main branch of $\VT$, $\M$ is suitable and $(\VT^*)^\frown \M$ is a $F$-quasi iteration of $\P$.
\end{enumerate} 
 For $\a=1$ we let $\Sigma^F_\a=\Lambda^F_\P$. Clearly, $\Sigma^F_1$ satisfies clauses 1-4 above. Suppose now we have defined $\la \Sigma^F_\a: \a<\b\ra$. For $\a<\b$, let $I_\a$ be the set of pairs $(\VT, \M)$ such that  $\VT$ is according to $\Sigma^F_\a$, it has $\a$ many rounds, $\M$ is the last model of $\VT$, $\M$ is $\Gamma$-suitable and if $\a$ is limit then $\M$ is the direct limit of the models along the main branch. Suppose now $\b=\nu+1$ for some $\nu$.  Then $\Sigma^F_{\b}(\VT)=b$ iff either
 \begin{enumerate}
 \item $\VT$ is a run of $\mathcal{G}(\P, \nu,\lambda^+)$, $\VT$ is according to $\Sigma^F_\nu$ and $\Sigma^F_\nu(\VT)=b$ or
 \item $\VT$ is of the form $\VU^\frown\M^\frown \U$ such that $(\VU, \M)\in I_\nu$, $\U$ is according to $\Lambda^F_\M$ and $\Lambda^F_\M(\U)=b$.
 \end{enumerate}
 Again, clearly $\Sigma^F_\b$ has the desired properties. 
 
 Lastly, suppose $\b$ is limit. We have that $\la \Sigma^F_\a: \a<\b\ra$ satisfies 1-4. It is then enough to show that clause 4 holds. To see this, fix $\VT=\la \T_\a, \M_\a:\a<\b\ra$ as in clause 4. We have that for each $\a<\b$, $(\VT\rest \a)^*$ is an $F$-quasi iteration of $\P$. What we need to show is that the direct limit along the main branch of $\VT$ is the same as the direct limit as defined by the two clauses of \rdef{quasi iteration}. Let then, for $\a<\gg<\b$, $\pi_{\a, \gg}:\M_\a\rightarrow \M_\gg$ be the iteration embeddings given by $\VT$. We have that for each $f\in F$ and $\a<\gg<\b$, $\pi_{\a, \gg}(f(\M_\a))=f(\M_\gg)$ and that $\M_\a=\cup_{f\in F} H_f^{\M_\a}$. Let then $\M$ be the direct limit of $\M_\a$'s under the maps $\pi_{\a, \gg}$ and let $\pi_\a:\M_\a\rightarrow \M$ be the direct limit embedding. We then have that $\M=\cup_{\a<\b, f\in F}\pi_0(H_f^\P)$ and for every $f\in F$, $\pi_0(H_f^\P)$ is the direct limit under $\pi^{f, \b}_{\a, \gg}$ where $\la \pi^{f, \xi}_{\a, \gg}: \a<\gg\leq \xi \leq \b\ra$ are the embeddings of $\VT^*$. It then follows that $\M$ is indeed the model constructed by the two clauses of \rdef{quasi iteration}.
\end{proof}

Suppose $F$ is a qsjs and $\P\in S(\Gamma)$ witnesses it. We then let $\Sigma^F_\P$ be the $(\lambda^+, \l^+)$-iteration strategy given by \rlem{getting strategies from qsjs}. Next we show that $\Sigma^F_\P$ has branch condensation.

\begin{lemma}[ZF]\label{getting branch condensation} Suppose $F$ is a qsjs and $\P\in S(\Gamma)$ witnesses it. Then $\Sigma^F_\P$ is a $(\l^+, \l^+)$-iteration strategy which is $\Gamma$-fullness preserving and has branch condensation.
\end{lemma}
\begin{proof} We only need to show that $\Sigma^F_\P$ has branch condensation as the rest follows from \rlem{getting strategies from qsjs}. Let then $(\R, \VT, \Q, \VU, c, \pi)$ be as in \rdef{branch condensation} where we let $M=\P$ and $\Sigma=\Sigma^F_\P$. We need to see that $c=\Sigma^F_\P(\VU)$. Because we have that $\pi:\Q\rightarrow \R$ and $i^{\VT}=\pi\circ i^{\VU}_c$ we get that, using clause 5 of \rdef{qsjs}, that $\Q$ is $\Gamma$-suitable and $i^\VU_c(f(\P))=f(\Q)$ for all $f\in F$. Let then $\la \U_\a, \M_\a: \a\leq \eta\ra$ be the normal components of $\VU$. Because $\VU\rest \eta$ is via $\Sigma^F_\P$,  we have that $i_{0, \eta}^{\VU}(f(\P))=f(\M_\eta)$ for every $f\in F$. It then follows that $i^{\U_\eta}_c(f(\M_\eta))=f(\Q)$ for all $f\in F$. Hence, $\U_\eta$ is according to $\Lambda^F_{\M_\eta}$, implying that $\Sigma^F_\P(\VU)=c$.
\end{proof}

We finish with the following lemma whose proof we leave to the reader as it is very close to the proof of \rlem{getting strategies from qsjs}.

\begin{lemma} Suppose $F$ and $G$ are two qsjs such that $F\subseteq G$. Then for any $\P\in S(\Gamma)$ witnessing that $F$ is qsjs, $\P$ witnesses that  $G$ is also qsjs and $\Sigma^F_\P=\Sigma^G_\P$.  
\end{lemma}

\subsection{$(\omega, \Gamma)$-suitable premice}

In this paper, our primary tool for constructing iteration strategies with branch condensation will be \rlem{getting strategies from qsjs} which heavily relies on clause 4 and 5 of \rdef{qsjs}. Here we develop some notions that we will later use to show that various $F$'s area qsjs and in particular, satisfy clause 4 and 5 of \rdef{qsjs}.

We continue with the set up of the previous sections. Recall that we have fixed a cardinal $\l$ and $\Gamma\subseteq \powerset(\powerset(\l))$. The basic notion we will need is that of $(\omega, \Gamma)$-suitable premice. These are formed by stacking $\omega$ many $\Gamma$-suitable premice and hence, they all have $\omega$-Woodin cardinals. 

\begin{definition}\label{omega suitable} A premouse $\P$ is $(\omega, \Gamma)$-suitable if there is an increasing sequence of $\P$-cardinals $\la \d_i: i<\omega\ra$ such that letting $\d_\omega=\sup_{i<\omega}\d_i$,
\begin{enumerate}
\item $\P\models ``\d_\omega$ is the largest cardinal",
\item for each $i$, $\d_i$ is a Woodin cardinal in $\P$,
\item if $\P_i=\P|(\d_i^{+\omega})^\P$ then $\P_0$ is a $\Gamma$-suitable premouse and $\P_{i+1}$ is a $\Gamma$-suitable premouse over $\P_i$.
\end{enumerate}
\end{definition}

If $\P$ is $(\omega, \Gamma)$-suitable then we let $\d_i^\P=\d_i$. We say $\P$ is an \textit{anomalous} $(\omega, \Gamma)$-suitable premouse if $\rho_\omega(\P)<\d_\omega^\P$. 

Suppose now $\P$ is any premouse with exactly $\omega$-Woodin cardinals. Let $\la \d_i : i\in [-1, \omega)\ra$ be such that $\d_{-1}=\emptyset$ and $\la \d_i: i<\omega\ra$ enumerates the Woodin cardinals of $\P$ in increasing order.  Let $\d_\omega=\sup_{i<\omega}\d_i$. Suppose $\P\models ``\d_{\omega}$ is the largest cardinal". Let $\T$ be a normal tree on $\T$ constructed via $\mathcal{G}(\P, \l^+)$. Notice that there is a natural way of rearranging $\T$ so that it is constructed via a run of $\mathcal{G}(\P, \omega+1, \l^+)$. In this rearrangement of $\T$, if $\M_n$ is the model of the beginning of the $n$th round and $n\leq\omega$ then the iteration embedding $i:\P\rightarrow \M_n$ exists. We have that $\M_0=\P$. Furthermore, if $n<\omega$ and $\T_n$ is the tree played in the $n$th round then $\T_n$ is based on the window $(i(\d_{n-1}), i(\d_n))$. The tree played in the $\omega$th round is played on $\M_\omega$ and is above $i(\d_\omega)$. Notice that not every normal $\T$ will be constructed in exactly $\omega+1$ non-trivial rounds (where we say that $n$th round is trivial if $I$ doesn't play any extender from the window $(i(\d_{n-1}), i(\d_n))$). 

In what follows, whenever we have a premouse like $\P$ above (and these include all $(\omega, \Gamma)$-suitable premice), we will think of normal iteration trees on $\P$ as stacks produced via a run of $\mathcal{G}(\P, \omega+1, \l^+)$. Given such a $\P$ and a normal tree $\T$ on $\P$, we say $\la \T_n, \M_n: n<k\ra$ are the normal components of $\T$ where $k\leq \omega+1$ is the least such that for all $m\geq k$, $\T_m$ is undefined. We say $\M$ is \textit{the last model} of $\T$ if 
 \begin{center}
   $\M= 
     \begin{cases}
      \M_\a^{\T_n}&: k=n+1, \a+1=lh(\T_n)\\
      \M_\omega  &: k=\omega\ \text{and}\ \M_\omega\ \text{is the direct limit of}\ \M_n\text{'s}\\
      \text{undefined} &: otherwise.
     \end{cases}$
\end{center}

\begin{definition} Suppose $\P$ is $(\omega, \Gamma)$-suitable. We say $\Sigma$ is an $(\omega, \Gamma)$-fullness preserving strategy for $\P$ if $\Sigma$ is an $(\omega^2, \l^++1)$-iteration strategy such that whenever $\VT=\la \T_\a, \M_\a: \a< \b\ra$ is a stack on $\P$ according to $\Sigma$ the following holds:
\begin{enumerate}
\item $\b\leq \omega^2$,
\item $\M_0=\P$  and for $\a<\b$, $\M_\a$ is $(\omega, \Gamma)$-suitable,
\item for $\a<\b$, letting $\la \T^\a_i, \M^\a_i: i< k_\a\ra$ be the normal components of $\T_\a$ we have that 
\begin{enumerate}
\item $k_\a\leq \omega$,
 \item if $\a+1<\b$ then $\M_{\a+1}$ is the last model of $\T_\b$,
 \item for each $i<k_\a$ and limit $\xi<lh(\T^\a_i)$, $\Q(\T^\a_i\rest \xi)$ exists and 
 \begin{center}
 $\Q(\T^\a_i\rest \xi)\insegeq \W^\Gamma(\M(\T^\a_i\rest \xi))$,
 \end{center}
\end{enumerate} 
\item if $\b=\omega^2$ and $\M_{\omega^2}$ is the direct limit of $\M_\a$'s under the iteration embeddings given by $\VT$ then $\M_{\omega^2}$ is $(\omega, \Gamma)$-suitable.
\end{enumerate}
\end{definition}

Next we need to introduce \textit{simultaneous genericity iterations} which we will use to show that the strategies we construct have branch condensation. Simultaneous genericity iteration were used for this purpose in \cite{ATHM} as well. First, however, we need \textit{generic genericity iterations}. 

Suppose $N\in H_{\l^+}$ and $\M\in H_{\l^+}$ is a premouse with a Woodin cardinal $\d$. We say $\T$ is \textit{the generic $N$-genericity tree} on $\M$ if $\T$ is a run of $\mathcal{G}(\M, \l^+)$ in which $I$ plays as follows: at stage $\a$, $E_\a^\T$ is the least extender such that for some $p\in Coll(\omega, N)$, $p$ forces that if $x$ is the generic code of $N$ then $x$ violates some axiom generated by $E_\a^\T$. If $\Sigma$ is a $\l^++1$-strategy for $\M$ then it can be shown that generic $N$-genericity iterations terminate and produce an iteration $i: \M\rightarrow \Q$ such that whenever $g\subseteq Coll(\omega, N)$ is generic and $x$ is the generic code of $N$ then $x$ is generic over $\Q$ for the extender algebra at $i(\d)$. The proof of this fact is just like the proof of the same fact for the usual genericity iterations (see \cite{EA} or \cite{OIMT}).

\begin{definition}\label{triangle sequence} We say $\la \R_i, \Q_i, m_i, \sigma_i, \nu_i, \Sigma: i<\omega\ra$ is a $\Gamma$-sequence of triangles with direct limit $\R_\omega$ if
\begin{enumerate}
\item for $i\leq \omega$, $\R_i$ and $\Q_i$ are $(\omega, \Gamma)$-suitable premice,
\item $m_i:\R_i\rightarrow_{\Sigma_1} \R_{i+1}$, $\sigma_i: \Q_i \rightarrow_{\Sigma_1} \R_{i+1}$ and $\nu_i: \R_i\rightarrow_{\Sigma_1} \Q_i$ are such that $m_i=\sigma_i\circ \nu_i$,
\item $\R_\omega$ is the direct limit of $\R_i$'s under the embeddings $m_{i, j}=_{def}m_{j-1}\circ m_{j-2}\circ \cdot\cdot\cdot m_i$,
\item $m_\omega$, $\sigma_\omega$, and $\nu_\omega$ are undefined, and
\item $\Sigma$ is an $(\omega, \Gamma)$-fullness preserving strategy for $\R_\omega$.
\end{enumerate}
\end{definition}

Suppose $\la \R_i, \Q_i, m_i, \sigma_i, \nu_i, \Sigma: i\leq \omega\ra$ is a $\Gamma$-sequence of triangles. Then we let $m_{i, \omega}:\R_i\rightarrow \R_\omega$ be the direct limit embedding and $\sigma_{i, \omega}=m_{i+1, \omega}\circ \sigma_i$. We then let $\Sigma^i$ be $m_{i, \omega}$-pullback of $\Sigma$ and $\Lambda^i$ be $\sigma_{i, \omega}$-pullback of $\Sigma$. 

\begin{definition}[Simultaneous genericity iterations]\label{simultaneous genericity iterations} Suppose $\la \R_k, \Q_k, m_k, \sigma_k, \nu_k, \Sigma: k< \omega\ra$ is a $\Gamma$-sequence of triangles with direct limit $\R_\omega$ and $\vec{N}=\la N_k: k<\omega \ra\subseteq H_{\l^+}$. Suppose further that either $\l=\omega$ and $\Sigma$ is an $(\omega^2, \omega_1+1)$-strategy or that $\l>\omega$ and for each $i$, $N_k\in H_\l$. We say $\la \R^j_k, \Q^j_k, \VS^j_k, \VW^j_k, m^j_k, \sigma^j_k, \nu^j_k: j, k\leq\omega\ra$ is the simultaneous $\vec{N}$-genericity iteration of $\la \R_k, \Q_k: i<\omega\ra$ via $\Sigma$ if the following holds:  
\begin{enumerate}
\item $\la \R^0_k, \Q^0_k, m^0_k, \sigma^0_k, \nu^0_k: k<\omega\ra=\la \R_k, \Q_k, m_k, \sigma_k, \nu_k: k<\omega\ra$,
\item for all $j, k<\omega$, $\VS^j_k$ is a non-dropping stack of finite length on $\R^j_k$ based on the window $(\d_{j}^{\R^j_k}, \d_{j+1}^{\R^j_k})$, $\VS^j_k$ is according to $\Sigma^k_{\R^j_k, \oplus_{l<j}\VS^l_k}$, and $\R^{j+1}_k$ is the last model of $\VS^j_k$,
\item for all $j, k< \omega$, $\VW^j_k$ is a non-dropping stack of finite length on $\Q^j_k$ based on the window $(\d_{j}^{\Q^j_k}, \d_{j+1}^{\Q^j_k})$, $\VW^j_k$ is according to $\Lambda^k_{\Q^j_k, \oplus_{l<j}\VW^l_k}$, and $\Q^{j+1}_k$ is the last model of $\VW^j_k$,
\item for $j, k<\omega$, $\sigma^j_k:\Q^j_k\rightarrow \R^{j}_{k+1}$, $m^j_k:\R^j_k\rightarrow \R^{j}_{k+1}$, $\nu^j_k:\R^j_k\rightarrow \Q^j_k$, and
\begin{center}
$m^j_k=\sigma^j_k\circ \nu^j_k$,
\end{center}
\item for each $j$, $\VS^j_0$ is the tree of generic $N_j$-genericity iteration of $\R^j_0$ in which $II$ plays according to $\Sigma^0_{ \R^j_0, \oplus_{l<j}\VS^l_o}$, 
\item for each $k, j<\omega$, letting $\VW^*=\nu^j_k\VS^j_k$, $\VW^j_k=\VW^{*\frown}\W$ where, letting $\M$ be the last model of $\VW^*$, we have that $\W$ is the tree of generic $N_j$-genericity iteration of $\M$ which is based on the window $(\d_{j}^\M, \d_{j+1}^\M)$ and is according to $\Lambda^k_{\M, (\oplus_{l<j}\VW^l_k)^\frown \VW^*}$,
\item for each $k, j<\omega$, letting $\VS^*=\sigma^j_k\VW^j_k$, $\VS^j_k=\VS^{*\frown}\S$ where, letting $\N$ be the last model of $\VS^*$, we have that $\S$ is the tree of generic $N_j$-genericity iteration of $\N$ which is based on the window $(\d_{j}^\N, \d_{j+1}^\N)$ and is according to $\Sigma^{k+1}_{\N, (\oplus_{l<j}\VS^l_k)^\frown \VS^*}$,
\item keeping the notation of clause 6 and 7, for each $k, j<\omega$, letting $s^j_k:\R^{j+1}_k\rightarrow \M$ and $w^j_k:\Q^{j+1}_k\rightarrow \N$ be the maps coming from the copying constructions, we have that $\nu^{j+1}_k=i^{\W}\circ s^j_k$, $\sigma^{j+1}_k=i^\S\circ w^j_k$ and $m^{j+1}_k=\sigma^{j+1}_k\circ \nu^{j+1}_k$.
\end{enumerate}
\end{definition} 

Suppose $\la \R^j_k, \Q^j_k, \VS^j_k, \VW^j_k, m^j_k, \sigma^j_k, \nu^j_k: j, k<\omega\ra$ is as in \rdef{simultaneous genericity iterations}. Then we say $\la \R^\omega_k, \Q^\omega_k, \VS^\omega_k, \VW^\omega_k, m^\omega_k, \sigma^\omega_k, \nu^\omega_k: k<\omega\ra$ is the direct limit of $\la \R^j_k, \Q^j_k, \VS^j_k, \VW^j_k, m^j_k, \sigma^j_k, \nu^j_k: j, k<\omega\ra$ if 
\begin{enumerate}
\item for each $k\leq \omega$, $\R^\omega_k$ is the direct limit of $\R^j_k$'s under $i^{\VS^j_k}$'s,
\item for each $k< \omega$, $\Q^\omega_k$ is the direct limit of $\Q^j_k$'s under $i^{\VW^j_k}$'s,
\item for each $k<\omega$, $\sigma^\omega_k$, $\nu^\omega_k$ and $m^\omega_k$ come from direct limit constructions. \end{enumerate}
Notice that we have that for each $k<\omega$, $\sigma^\omega_k:\Q^\omega_k\rightarrow \R^{\omega}_{k+1}$, $m^\omega_k:\R^\omega_k\rightarrow \R^{\omega}_{k+1}$, $\nu^\omega_k:\R^\omega_k\rightarrow \Q^\omega_k$, and
\begin{center}
$m^\omega_k=\sigma^\omega_k\circ \nu^\omega_k$,
\end{center}
We also let $\R^\omega_\omega$ be the direct limit of $\R^\omega_k$'s under $m^\omega_k$'s. We say $\R^\omega_\omega$ is the direct limit of $\la \R^j_k, \Q^j_k, \VS^j_k, \VW^j_k, m^j_k, \sigma^j_k, \nu^j_k: j, k\leq \omega\ra$. Notice that, by the copying construction, $\R^\omega_\omega$ is a $\Sigma$-iterate of $\R_\omega$ and the length of the stack producing the iteration is $\omega^2$.

\subsection{Review of $\H$ analysis}

In this section, we review $\H$ analysis of models satisfying $AD^++V=L(\powerset(\mathbb{R}))+MC+\Theta=\theta_0$. Until the end of this subsection, we assume $V$ satisfies the above theory. We let $\Gamma=\powerset(\mathbb{R})$. For the duration of this subsection, we will drop $\Gamma$-from our notation. Thus, a suitable premouse is a $\Gamma$-suitable premouse and etc.

Suppose $\P$ is suitable and $A\subseteq \bR$ is $OD$. We say $\P$ \textit{weakly term captures} $A$ if letting $\d=\d^\P$, for each $n<\omega$ there is a term relation $\tau\in \P^{Coll(\omega, (\d^{+n})^\P}$ such that for comeager many $\P$-generics, $g\subseteq Coll(\omega, (\d^{+n})^\P)$, $\tau_g=\P[g]\cap A$. We say $\P$ \text{term captures} $A$ if the equality holds for all generics. The following lemma is essentially due to Woodin and the proof can be found in \cite{CMI}.

\begin{lemma} Suppose $\P$ is suitable and $A\subseteq \bR$ is $OD$. Then $\P$ weakly term captures $A$. Moreover, there is a suitable $\Q$ which term captures $A$.
\end{lemma}

Given a suitable $\P$ and an $OD$ set of reals $A$, we let $\tau_{A, n}^\P$ be the standard name for a set of reals in $\P^{Coll(\omega, (\d^{+n})^\P}$ witnessing the fact that $\P$ weakly captures $A$. We then define $f_A\in F(\Gamma)$ by letting
\begin{center}
$f_A(\P)=\la \tau^\P_{A, n} : n<\omega\ra$.
\end{center}
Let $F_{od}=\{ f_A: A\subseteq \bR \wedge A\in OD\}$. 

 All the notions we have defined in \rsec{suitable premouse}, \rsec{quasi iterability section} and \rsec{f-guided strategies} can be redefined for ordinal definable sets $A\subseteq \mathbb{R}$ using $f_A$ as the relevant function. To save some ink, in what follows, we will say $A$-iterable instead of $f_A$-iterable and similarly for other notions. Also, we will use $A$ in our subscripts instead of $f_A$.

 The following lemma is one of the most fundamental lemmas used to compute $\H$ and it is originally due to Woodin. Again, the proof can be found in \cite{CMI}.

\begin{theorem}\label{existence of quasi-iterable premice} For each $f\in F_{od}$, there is $\P\in S(\Gamma)$ which is $(F_{od}, f)$-quasi iterable.
\end{theorem}

Let $\M_\infty=\M_{\infty, F_{od}}$.

\begin{theorem}[Woodin, \cite{CMI}]\label{hod theorem} $\d^{\M_\infty}=\Theta$, $\M_\infty\in \H$ and $\M_\infty|\Theta=(V_\Theta^\H, \vec{E}^{\M_\infty|\Theta}, \in)$.
\end{theorem}

Finally, if $a\in H_{\omega_1}$, then we could define $\M_\infty(a)$ by working with suitable premice over $a$. Everything we have said about suitable premice can also be said about suitable premice over $a$ and in particular, the equivalent of \rthm{hod theorem} can be proven using $\H_{a\cup\{a\}}$ instead of $\H$ and $\M_\infty(a)$ instead of $\M_\infty$. 

\section{The proof of \rthm{main theorem}}

The rest of this paper is devoted to the proof of \rthm{main theorem}. From now on $\k$ is as in the hypothesis of \rthm{main theorem}. Given $a\subseteq V_\kappa$, let
\begin{center}
$Lp(a)=\cup\{ \N : \N$ is a sound countably iterable mouse over $a$ such that $\mathcal{\rho}_\omega(\N)=a\}$. 
\end{center}
We define $\la Lp_\xi(a) : \xi<\k^+\ra$ by the following recursion:
\begin{enumerate}
\item for $\a<\k^+$, $Lp_{\a+1}=Lp_1(Lp_\a(a))$,
\item for $\l<\k^+$ limit, $Lp_\l(a)=\cup_{\a<\l}Lp_\a(a)$.
\end{enumerate}

 
We will use the following fundamental result throughout this paper. 

\begin{theorem}[Schimmerling-Zeman, \cite{SchZem}]\label{SchZem} For all $A\subseteq V_\k$, $Lp(A)\models \square_\k$.
\end{theorem}

From now on we fix $A\subseteq \kappa$ such that $A$ codes $V_\kappa$.
Because $\neg\square_\kappa$, we must have that $o(Lp(A))<\kappa^+$.
Let $\xi=o(Lp(A))$. Because $\kappa$ is singular, we have that $\cf(\xi)<\k$. 

\begin{definition} We say $\mu<\kappa$ is a \textit{good} point if $\cf(\k), \cf(\xi)<\mu$, $\mu$ is regular and $\mu^{\omega}=\mu$.
\end{definition}

Clearly there are good points. Suppose then that $\mu$ is a good point and $g\subseteq Coll(\omega, \mu)$ is generic. Working in $V[g]$, we let
\begin{center}
$\Gamma_{\mu, g}=\{ B\subseteq \bR : L(B, \bR)\models AD^+\}$.
\end{center}
The following theorem shows that $\Gamma_{\mu, g}$ is not empty. Its proof is essentially Steel's proof that $L(\bR)\models AD$ (see \cite{PFA}). See \cite{ATHM} for the definition of $L^{\Sigma}(\bR)$. It is the minimal $\Sigma$-mouse over $\bR$ which contains all the ordinals and has no extenders on its sequence\footnote{Below we stop feeding $\Sigma$ after stage $\k^+$.}. 

\begin{theorem}\label{strategies are determined} Suppose $\mu$ is a good point and $g\subseteq Coll(\omega, \mu)$ is generic. Suppose in $V[g]$, $\M$ is a $\k^+$-iterable countable mouse over some set $X$ and that $\rho_\omega(\M)=X$. Let $\Sigma$ be the $\k^+$-iteration strategy of $\M$. Then $L^\Sigma(\bR)\models AD^+$. Hence, in $V[g]$, letting $\Lambda=\Sigma\rest H_{\omega_1}$, 
\begin{center}
$Code(\Lambda)\in \Gamma_{\mu, g}$.
\end{center}
\end{theorem}

Our goal is to show that for some good $\mu$ and generic $g\subseteq Coll(\omega, \mu)$, there is $B\in \Gamma_{\mu, g}$ such that
\begin{center}
$L(B, \bR)\models AD^++\Theta=\theta_1$.
\end{center} 
We can then use \rthm{equiconsistency} and the homogeneity of the forcing to prove \rthm{main theorem}.
Let then (*) be the following statement: For any good $\mu$ and a generic $g\subseteq Coll(\omega, \mu)$, in $V[g]$, there is no inner model $M$ such that $Ord, \bR\subseteq M$ and
\begin{center}
$M\models AD^++\Theta=\theta_1$\ \ \ \ \ \ \ (*).
\end{center}
Towards a contradiction we assume that (*) holds. 

Borrowing some lemmas from the next subsection we can characterize sets in $\Gamma_{\mu, g}$ in terms of $\bR$-mice. 

\begin{theorem}\label{capturing by fully iterable mice} Suppose $\mu$ is good and $g\subseteq Coll(\omega, \mu)$ is generic. Let $B\in \Gamma_{\mu, g}$. Then, in $V[g]$, there is a sound $\bR$-mouse $\N$ such that $\rho_\omega(\N)=\bR$, $B\in \N$ and countable submodels of $\N$ are $\k^+$-iterable.
\end{theorem}
\begin{proof} Fix $g$ and $B$ as in the hypothesis. We work in $V[g]$. We have that $L(B, \bR)\models MC+\Theta=\theta_0$. It follows from the main theorem of \cite{V=K(R)} that in $L(B, \bR)$ there is a countably iterable sound $\bR$-mouse $\N$ such that $\rho_\omega(\N)=\bR$ and $B\in \N$. Fix such a mouse $\N$. It is then enough to show that, in $V[g]$, countable submodels of $\N$ are $\k^+$-iterable. To see this, let $\pi :\S\rightarrow \N$ be a countable submodel of $\M$. We have that $L(B, \bR)\models ``\S$ is $\omega_1$-iterable". It then follows from clause 2 of \rlem{extending strategies via suitability} (in particular, see clause 2c and 2d) that $\S$ is $\k^+$-iterable. 
\end{proof}

\begin{definition} Suppose $\mu$ is good and $g\subseteq Coll(\omega, \mu)$ is generic. Working in $V[g]$, we let $\S^-_{\mu, g}$ be the union of those sound $\bR$-mice $\N$ such that $\rho_\omega(\N)=\bR$, countable submodels of $\N$ are $\k^+$-iterable and there is a set of reals $B\in \Gamma_{\mu, g}$ such that $B$ codes $\N$. We let $\S_{\mu, g}=L(\S^-_{\mu, g})$.
\end{definition}

\begin{theorem}\label{determinacy in the max model} Suppose $\mu$ is good and $g\subseteq Coll(\omega, \mu)$ is generic. Then in $V[g]$, $\S_{\mu, g}\models AD^++\theta_0=\Theta$.
\end{theorem}
\begin{proof}
It follows from the previous theorem that $\powerset(\bR)^{\S^-_{\mu, g}}=\Gamma_{\mu, g}$. Hence, $\S^-_{\mu, g}\models AD^+$. Also, notice that $\S^-_{\mu, g}\models \theta_0=\Theta$. This is because otherwise there is $B\in \Gamma_{\mu, g}$ such that $w(B)=\theta_0^{\S_{\mu, g}}$. It then follows that $L(B, \bR)\models \theta_0<\Theta$ which contradicts (*). But now, using the scales analysis of \cite{ScalesK(R)} and \cite{Scalesendgap} and the core model induction of \cite{PFA}, we get that $\S_{\mu, g}\models AD^+$. It then follows from (*) that $\S_{\mu, g}\models AD^++\theta_0=\Theta$.
\end{proof}

 Suppose $\mu$ is good.  We let $\Theta^{\mu}=\Theta^{S_{\mu, g}}$ and if $a\in H_{\mu^+}$ then we let $\P_{\mu, a}=(\M_\infty(a))^{S_{\mu, g}}$ where $g\subseteq Coll(\omega, \mu)$ is some generic. Notice that $\P_{\mu, a}$ is independent of $g$ and $\P_{\mu, a}\in V$. We let $\P_\mu=\P_{\mu, \emptyset}$. Given a good hull $(M, \pi)$ at $\mu$ such that $a$ is in $M[g]$ and is countable there, we let $\P_{\mu, a}^M=\pi^{-1}(\P^M_{\mu, a})$ and $\S^M_{\mu, g}=(L(\pi^{-1}(\S_{\mu, g})))^{M[g]}$.

\subsection{Good points and good hulls}

Clearly there are good points. Let $\mu$ be a good point. Recall that $\xi=o(Lp(A))$. Let $\nu=\cf(\xi)$ and let $f:\nu\rightarrow \xi$ be an increasing cofinal function. Let $\zeta=\kappa^{+\omega}$. 

\begin{definition}
We say $(M, \pi)$ is a \textit{good hull at $\mu$} if $\pi: M\elesub V_{\zeta}$ is such that 
$\mu+1\subseteq M$, $\card{M}=\mu$, $M^\omega\subseteq M$ and $\{A, f\}\in ran(\pi)$.
\end{definition}
An easy Skolem hull argument shows that there are good hulls at $\mu$. If $(M, \pi)$ is a good hull at $\mu$ then we let $\kappa_{M, \pi}=\pi^{-1}(\kappa)$ and $A_{M, \pi}=\pi^{-1}(A)$. Often times, when it is clear what $\pi$ is, we will omit it from subscripts. The fact that $M$ is countably closed implies that $M$ is \textit{full} with respect to countably iterable mice. The proof of this lemma is essentially the covering argument.

\begin{lemma}\label{fullness of M at A} Suppose $\mu$ is good and $(M, \pi)$ is a good hull at $\mu$. Then $Lp(A_M)\in M$.
\end{lemma}
\begin{proof} We only outline the proof of this well-known fact. Suppose not. Let $\xi_M=\pi^{-1}(\xi)$. Let $\M\insegeq Lp(A_M)$ be the least such that $\rho_\omega(\M)=A_M$ and $\M\not \in M$. Let $E$ be the $(\k_M, \k)$-extender from $\pi$ and let $\M^*=Ult(\M, E)$. Then $\M^*$ is countably iterable (because $E$ is countably closed) and is sound $A$-mouse such that $\rho_\omega(\M^*)=A$. Hence, $\M^*\insegeq Lp(A)$. But because $\pi\rest \xi_M$ is cofinal in $\xi$, we get that $Lp(A)\inseg \M^*$, contradiction.
\end{proof}

The following lemma is our main tool for extending iteration strategies to $\k^+$-iteration strategies. It is essentially due to Steel (see Lemma 1.25 of \cite{PFA}) and we leave the proof to the readers. 

\begin{lemma}\label{extending kappa strategies} Suppose $\mu$ is good and $g\subseteq Coll(\omega, \mu)$ is generic. Then the following holds.
\begin{enumerate}
\item Suppose $\M$ is a sound premouse over some set $X\in V_\k[g]$ such that $\rho_\omega(\M)=X$ and in $V[g]$, $\M$ is $\k$-iterable. Then $\M$ is $\kappa^+$-iterable.
\item Suppose $B\in \Gamma_{\mu, g}$ and $\Gamma\subset (\utilde{\Delta}^2_1)^{L(B, \bR)}$ is a good pointclass. Suppose $(\P, \Lambda)\in V[g]$ is such that in $V[g]$, $\P$ is countable and $\Gamma$-suitable, and $\Lambda$ is a $\Gamma$-fullness preserving $(\omega_1, \omega_1)$-iteration strategy for $\P$ with branch condensation. Then $\Lambda$ can be extended to a $(\k^+, \k^+)$-iteration strategy which acts on non-dropping trees and has the branch condensation. 
\end{enumerate}
\end{lemma} 

Let $\Sigma$ be the extension of $\Lambda$ mentioned in clause 2 of \rlem{extending kappa strategies}. Then notice that it follows from branch condensation that if $\VT$ is a stack according to $\Sigma$ and $\pi: N\rightarrow V_{\zeta}[g]$ is such that $N$ is countable in $V[g]$ and $\VT\in rng(\pi)$ then $\pi^{-1}(\VT)$ is according to $\Lambda$. Moreover, whenever $\Q\in I(\P, \Sigma)$, $\eta$ is a cutpoint of $\Q$ and $\M$ is a mouse over $\Q|\eta$ such that, in $V[g]$, its countable hulls have iteration strategy in $\Gamma$, then $\M\insegeq \Q$.

The following lemma will be used to show that good hulls are correct about $\k^+$-iterable mice. 

\begin{lemma}\label{extending strategies via suitability} Suppose $\mu$ is good and $g\subseteq Coll(\omega, \mu)$ is generic. Working in $V[g]$, suppose $B\in \Gamma_{\mu, g}$ and let $\Gamma^*=\powerset(\bR)\cap L(B, \bR)$.  Suppose that for some countable set $X$, $\M\inseg \W^{\Gamma^*}(X)$ is such that $\rho_\omega(\M)=\omega$. Then, in $V[g]$, there are 
\begin{enumerate}
\item a good pointclass\footnote{Recall a good pointclass is a pointclass which is closed under existential real quantification, is $\omega$-parametrized and has the scale property.} $\Gamma\subset (\utilde{\Delta}^2_1)^{L(B, \bR)}$,
\item  a $\Gamma$-suitable $\P$ over $X$ which has an $(\omega_1, \omega_1)$-iteration strategy $\Lambda\in L(B, \bR)$ with branch condensation and a $(\k^+, \k^+)$-iteration strategy $\Sigma$ with branch condensation which acts on non-dropping trees
such that 
\begin{enumerate}
\item $L(B, \bR)\models ``\Lambda$ is $\Gamma$-fulness preserving",
\item $\Lambda$ and $\Sigma$ agree on $dom(\Sigma)\cap H_{\omega_1}$,
\item there is some $\a$ such that if $\N_\a$ is the $\a$th model of $(L[\vec{E}][X])^\P$-construction then $Core(\N_\a)=\M$,
\item for any $\nu<\k$ and for any $h\subseteq Coll(\omega, \nu)$, $\Sigma$ can be extended to a $(\k^+, \k^+)$-iteration strategy in $V[g*h]$  and
\item if $X\in V$ then $\P\in V$ and $\Sigma\rest V\in V$.
\end{enumerate} 
\end{enumerate}
\end{lemma}
\begin{proof}
Because $\M$ is iterable in $L(B, \bR)$, we can find a good pointclass $\Gamma\subset (\utilde{\Delta}^2_1)^{L(B, \bR)}$ such that $\M$ has an iteration strategy coded into $\Gamma$. Let $\vec{C}\in L(B, \bR)$ be a $\Gamma$-sjs\footnote{Recall that a $\Gamma$-sjs is a countable subset $S$ of $\Gamma$ such each $A\in S$ has a $\Gamma$-scale all of whose norms are in $S$, see \cite{CMI} or \cite{PFA} for more on sjs.}  and let $\Gamma_1\subset (\utilde{\Delta}^2_1)^{L(B, \bR)}$ be a good pointclass such that $\vec{C}\in \utilde{\Delta}_{\Gamma_1}$. Let $F:\bR\rightarrow L(B, \bR)$ be a function such that for cone of $x$, $F(x)=(\N^*_x, \M_x, \d_x, \Sigma_x)$ Suslin, co-Suslin captures $\Gamma_1$ (as in Theorem 1.2.9 of \cite{ATHM}). Let $x\in \bR$ be such that $X$ is coded by $x$, $F(x)$ is defined and $\vec{C}$ is Suslin, co-Suslin captured by $(\N^*_x, \M_x, \d_x, \Sigma_x)$. It follows that $\W^\Gamma(\N^*_x|\d_x)\in \N^*_x$. Because sjs's condense to themselves, by a Skolem hull argument we get that there are club many $\eta$ such that $\W^\Gamma(\N^*_x|\eta)\models ``\eta$ is Woodin". Let then $\eta$ be the least such that  $\W^\Gamma(\N^*_x|\eta)\models ``\eta$ is Woodin". It then follows that if $\P^-=(L[\vec{E}][X])^{\N^*_x|\eta}$ and $\P=Lp_\omega^\Gamma(\P^-)$ then $\P$ is $\Gamma$-suitable over $X$. 

Appealing to universality (see Lemma 1.1.26 of \cite{ATHM}), we get that $\M\inseg \P$ and that it is reached by $(L[\vec{E}][X])^\P$-construction of $\P$. Working in $V[g]$, let $\Lambda\in L(B, \bR)$ be the $(\omega_1, \omega_1)$-iteration strategy of $\P$ induced by $\Sigma_x$ (recall that background constructions inherit a strategy from the background as in \cite{FSIT}). It then follows that $L(B, \bR)\models ``\Lambda$ is $\Gamma$-fullness preserving" (one can use an argument like the one in Lemma 3.2.3 of \cite{ATHM}).  Using Steel's lifting techniques from \cite{PFA} which was summarized in clause 2 of \rlem{extending kappa strategies}, we can lift $\Lambda$ to a $(\kappa^+, \k^+)$-iteration strategy $\Sigma$ such that $\Sigma$ acts on trees with no drops.  We now have that $(\P, \Gamma, \Lambda, \Sigma)$ satisfies 1 and 2a-2c. 

To get 2d, we can use generic comparisons. First, we show that $\Sigma$ can be extended to a $(\k, \k)$-iteration strategy in $V[g*h]$ where for some $\nu<\k$, $h\subseteq Coll(\omega, \nu)$ is $V[g]$-generic. Working in $V[g]$, let $\pi: M\rightarrow V_{\zeta}[g]$ be a countable hull such that $(\P, \Sigma)\in rng(\pi)$. Let $\k_M=\pi^{-1}(\k)$ and let $\Sigma^M=\pi^{-1}(\Sigma)$. It is enough to show that 2d holds in $M$ for $\k_M$ and $\Sigma^M$. We have that $\Sigma^M=\Sigma\rest V_{\k_M}^M$. Let $\nu<\k_M$ and let $k\subseteq Coll(\omega, \nu)$ be $M$-generic. Let
\begin{center}
$D_{\nu, k}=\{ (a, \W^\Gamma(a)): a\in V_{\k_M}^{M[k]}\}$. 
\end{center}
Let $C=\{ (x, y) : x, y\in \mathbb{R}\wedge ``y$ codes an $\M\inseg\W^\Gamma(x)$ such that $\rho_\omega(\M)=x"\}$. Without loss of generality, we can assume that $\Lambda$ respects $C$. Notice that if $i:\P\rightarrow \Q$ comes from an iteration according to $\Lambda$ and $x\in \bR$ is generic over $\Q$ for $\mathbb{B}^\Q$ then 
\begin{center}
$\W^\Gamma(x)=\cup\{ \M : \M\in \Q[x]$ is coded by $y\in \bR^{\Q[x]}$ and $\Q[x]\models y\in \tau_C^\Q\}$ \ \ \ \ \ \ \ (1).
\end{center} 
For each $\eta<\k_M$ let $\P_\eta\in M$ be the $\Sigma^M$-iterate of $\P$ which is obtained via the $H_\eta^M$-generic genericity iteration. Let $i_\eta:\P\rightarrow \P_\eta$ be the iteration embedding according to $\Sigma^M$.

It then follows from (1) that $D_{\nu, k}\in M[k]$. This is because $(a, \N)\in D_{\nu, k}$ iff whenever $\eta\in (\nu, \k_M)$, $l\subseteq Coll(\omega, \eta)$ is $M$-generic, $x$ codes $a$ and
\begin{center}
 $\S=\cup\{ \M : \M\in \P_\eta[x]$ is coded by $y\in \bR^{\P_\eta[x]}$ and $\P_\eta[x]\models y\in i_\eta(\tau_C^{\P})\}$
\end{center}
then $\N$ is obtained from $\S$ via $S$-constructions (see \cite{ATHM} or \cite{SelfIter}). 

Because $D_{\nu, k}\in M[k]$ for all $\nu, k$, we can do generic comparisons in $M$. Suppose $\nu< \k_M$ is an $M$-cardinal and $k\subseteq Coll(\omega, \nu)$ is $M$-generic. Then we let $\S_\eta\in M$ be the $\Sigma^M$-iterate of $\P$ which is obtained by generically comparing all $\Gamma$-suitable $\Q$'s which are in $H_{\omega_1}^{M[k]}$ (see \cite{CMI} or \cite{PFA}). We have that $\S_\eta\in M$ and whenever $l\subseteq Coll(\omega, \nu)$ is $M$-generic then $\S_\eta$ is the result of generically comparing all $\Gamma$-suitable $\Q$'s which are in $H_{\omega_1}^{M[l]}$. Let $j_\eta:\P\rightarrow \S_\eta$ be the iteration map.

Fix now some $\nu<\k_M$ and let $h\subseteq Coll(\omega, \nu)$ be $M$-generic. Let $\VT\in H_{\omega_1}^{M[h]}$ be a stack on $\P$ which is according to $\Sigma$ such that the last component of $\VT$ is of limit length. We let $\Q$ be the last model of $\VT$. Then $\Q\in M[h]$ is $\Gamma$-suitable. We then have that $\Sigma(\VT)=b$ iff  $b$ is the unique branch of $\VT$ such that $\M^\T_b=\Q$ and there is $\sigma:\Q\rightarrow \S_\nu$ such that
\begin{center}
$j_\eta=\sigma\circ i^\VT_b$.
\end{center}
Let $\phi[\VT, \Q, b, \S_\nu]$ be the formula on the right side of the equivalence. It follows from absoluteness that $\Sigma(\VT)=b$ iff $M\models \phi[\VT, \Q, b, \S_\nu]$. Hence, $\Sigma\rest V_{\k_M}^{M[g]}\in M[g]$. Using the elementarity of $\pi$ we get that  $\Sigma$ can be extended to a $(\k, \k)$-strategy in $V[g*h]$ where for some $\nu<\k$, $h\subseteq Coll(\omega, \nu)$ is $V[g]$-generic. $\Sigma$ can then be extended to a $(\k^+, \k^+)$-strategy by using clause 2 of \rlem{extending kappa strategies}. We then get that $(\P, \Gamma, \Lambda, \Sigma)$ also satisfies 2d. It remains to show that 2e can also be satisfied.  

Notice that if $X\in V$ then $\M\in V$. Let then $\R=\M_\infty(\P, \Lambda)$. By homogeneity of the collapse, we have that $\R\in V$. Let now $(M, \pi)$ be a good hull at $\mu$ such that $(\R, \Sigma_\R)\in rng(\pi)$\footnote{It is shown in \cite{ATHM} that if $\P$ is $\Gamma$-suitable, $\Sigma$ is a $\Gamma$-fullness preserving iteration strategy for $\P$ with branch condensation and $\R\in I(\P, \Sigma)$ is an iterate of $\P$ via some $\VT$ then $\Sigma_{\R, \VT}$ is independent of $\VT$. This is the reason for omitting the stack from the subscript in $\Sigma_{\R}$. We will do this throughout the paper.} and $\P$ is countable in $M[g]$. Let $\S=\pi^{-1}(\R)$. Then $\S$ is a $\Lambda$-iterate of $\P$ and $\S\in V$. It follows that $(\S, \Gamma, \Lambda_\S, \Sigma_\S)$ satisfies 1 and 2a-2d. We claim that $(\S, \Sigma_\S)$ satisfies 2e.  The proof is just like the proof of 2d above using the fact that, by the homogeneity of the collapse, if $G$ is the function $a\rightarrow \W^\Gamma(a)$ then $G\rest V$ is definable in $V$. We leave the details to the reader.
\end{proof}

Suppose $\mu$ is good and $g\subseteq Coll(\omega, \mu)$ is generic. Working in $V[g]$, for $a\in V_\k[g]$ we let
\begin{center}
$\W_{\mu, g}(a)=\W^{(\powerset(\powerset(\k)))^{V[g]}}(a)$.
\end{center}

\begin{lemma}\label{equivalencies 1} Suppose $\mu$ is a good point and $g\subseteq Coll(\omega, \mu)$ is generic. Then the following holds.
\begin{enumerate}
\item For any $a\in H_{\omega_1}^{V[g]}$, $\W_{\mu, g}(a)= \W^{\Gamma_{\mu, g}}(a)$.
\item  Suppose $(M, \pi)$ is a good hull at $\mu$ and $a\in V_{\k_M}^{M[g]}$. Then
\begin{center}
$\W_{\mu, g}(a)\in M$.
\end{center} 
\item Suppose $\nu>\mu$ is also good and let $h\subseteq Coll(\omega, \nu)$ be $V[g]$-generic. Suppose $a$ is countable in $V[g]$. Then 
\begin{center}
$\W_{\mu, g}(a)=\W_{\nu, h}(a)$.
\end{center} 
\end{enumerate}
\end{lemma}
\begin{proof} 
\begin{enumerate}
\item  It follows from \rthm{strategies are determined} that $\W_{\mu, g}(a)\insegeq \W^{\Gamma_{\mu, g}}(a)$. Suppose then $\M\inseg \W^{\Gamma_{\mu, g}}(a)$ is such that $\rho_\omega(\M)=a$. Let $B\in \Gamma_{\mu, g}$ be such that $L(B, \bR)\models ``\M$ is $\omega_1$-iterable". It then follows from 2c of \rlem{extending strategies via suitability} that $\M$ is $\kappa^+$-iterable. Hence, $\M\inseg \W_{\mu, g}(a)$.
\item We work in $V[g]$. First it follows from \rlem{fullness of M at A} that $\W_{\mu, g}(A_M)\in M$. It follows from clause 1 that  for every $a\in V_{\k_M}^M$, $\W_{\mu, g}(a)\in \W_{\mu, g}(A_M)$. Hence, clause 2 follows.
\item It follows from part 1 that it is enough to show that $\W^{\Gamma_{\mu, g}}(a)=\W^{\Gamma_{\nu, h}}(a)$. Suppose then that $\M\inseg \W^{\Gamma_{\mu, g}}(a)$ is such that $\rho_\omega(\M)=a$. It follows from 2d of \rlem{extending strategies via suitability} that $\M$ is $\k^+$-iterable in $V[h]$. Hence, $\M\inseg \W^{\Gamma_{\nu, h}}(a)$. 

Conversely, suppose $\M\inseg \W^{\Gamma_{\nu, h}}(a)$ is such that $\rho_\omega(\M)=a$. It then follows from 2e of \rlem{extending strategies via suitability} that $\M$ is $\k^+$-iterable in $V[g]$ implying that $\M\inseg \W^{\Gamma_{\mu, g}}(a)$. 
\end{enumerate}
\end{proof}

Next we would like to show that good hulls capture $\k^+$-iterable mice, i.e., if $(M, \pi)$ is good and $a\in H_{\k_M}^M$ then 
$\W_{\mu, g}(a)=(\W_{\mu, g}(a))^{M[g]}$. For this we will need the following lemma which will come handy later on as well. 

\begin{lemma}\label{hod is short tree iterable} Suppose $\mu$ is good and let $g\subseteq Coll(\omega, \mu)$ be generic. Then in $V[g]$, $\P_{\mu}$ is $\powerset(\powerset(\k))$-short tree iterable.
\end{lemma}
\begin{proof}
Working in $\S_{\mu, g}$, let $\Q$ be a $\Gamma_{\mu, g}$-suitable premouse which is $F_{od}$-quasi-iterable. Let $(M, \pi)$ be a good hull at $\mu$ such that $\Q\in M$ and $M[g]\models ``\Q$ is countable". By elementarity of $\pi$, it is enough to show that 
\begin{center}
$M[g]\models ``\P_\mu^M$ is $\powerset(\powerset(\k_M))$-short tree iterable". 
\end{center}
We verify this only for trees. The proof for stacks is very similar and only notationally more involved. To show this for trees then it is enough to show that whenever $\T\in (H_{(\k^+_M)^M})^{M[g]}$ is a tree on $\P_\mu$ such that $M[g]\models ``\T$ is $\powerset(\powerset(\k_M))$-short" then there is $b\in M[g]$ such that $b$ is a cofinal well-founded branch of $\T$ such that $\Q(b, \T)$ exists and $\Q(b, \T)\trianglelefteq (\W_{\mu, g}(\M(\T)))^M$. Fix then such a tree $\T$ and let $\N=\Q(\T)$. Notice that $M$ is countable in $\S_{\mu, g}$ and in $\S_{\mu, g}$, $\P_{\mu}^M$ is an $F_{od}$-quasi-iterate of $\Q$. It then follows that
\begin{center}
$\S_{\mu, g}\models ``\P_{\mu}^M$ is $\Gamma_{\mu, g}$-short tree iterable"\ \ \ \ \ \ \ (1).
\end{center}  
It follows from elementarity of $\pi$ that $\pi(\N)$ is $\k^+$-iterable in $V[g]$ and hence, $\N$ is $\k^+$-iterable in $V[g]$. Therefore, using clause 1 of \rlem{equivalencies 1}, we get that 
\begin{center}
$S_{\mu, g}\models ``\N$ is $\omega_1$-iterable"\ \ \ \ \ \ (2).
\end{center}
Combining (1) and (2), we get a cofinal branch $b$ of $\T$ such that $\Q(b, \T)$ exists and $\N=\Q(b, \T)$. By absoluteness $b\in M[g]$.
\end{proof}

\begin{theorem}[Correctness of good hulls]\label{correctness of good hulls} Suppose $\mu$ is good and $g\subseteq Coll(\omega, \mu)$ is generic. Suppose $a\in V_{\k_M}^{M[g]}$. Then
\begin{center}
$\W_{\mu, g}(a)=\W_{\mu, g}(a)^{M[g]}$
\end{center} 
\end{theorem}
\begin{proof} 
Notice that because of $\pi$, $(\W_{\mu, g}(a))^M\insegeq \W_{\mu, g}(a)$. It is then enough to show that if $\M\inseg \W_{\mu, g}(a)$ is such that $\rho_{\omega}(\M)=a$ then $M\models ``\M$ is $\k_M^+$-iterable". Because of \rlem{extending kappa strategies}, it is enough to show that $M\models ``\M$ is $\k_M$-iterable". It follows from clause 2 of \rlem{equivalencies 1} that $\M\in M$. 

Next, we have that the set
\begin{center}
$D=\{ (x, \W_{\mu, g}(x)): x\in V_{\k_M}^M\}$
\end{center}
is $OD_{A_M}$ and hence, by $MC$\footnote{$MC$ stands for mouse capturing. It was shown in \cite{ATHM} that it holds in the minimal model of $AD_{\mathbb{R}}+``\Theta$ is regular". For more on $MC$ see \cite{ATHM} or \cite{DMATM}.} in $\S_{\mu, g}$ and by clause 1 of \rlem{equivalencies 1}, $D\in \W_{\mu, g}(A_M)$. Since $\W_{\mu, g}(A_M)\in M$, $D\in M$.  Fix then some $\eta<\k_M$ such that $a\in H_\eta^M$. Let $\Q_\eta$ be the result of a generic genericity iteration of $\P_{\mu}^M$ making $H_\eta$ generically generic. Notice that because $D\in M$ and $\P_\mu^M$ is $\Gamma_{\mu, g}$-short tree iterable we have that the entire iteration can be done in $M$. This iteration produces a $\Gamma_{\mu, g}$-correctly guided tree $\T$ with last model $\Q_\eta$ (where ``last model" is in the sense of \rdef{last model of finite stack}). 

Working in $S_{\mu, g}$, let $\Lambda$ be the $\omega_1$-iteration strategy of $\M$. We have that whenever $\Q$ is $\Gamma_{\mu, g}$-suitable and $\M$ is generic over $\Q$ then $\Lambda\rest H_{\omega_2}^{\Q[\M]}\in \Q[\M]$. It then follows that $\Q[\M]\models ``\M$ is $\omega_2$-iterable". We let $\Lambda^\Q=\Lambda\rest H_{\omega_2}^{\Q[\M]}$. Notice that $\Q[\M]$ satisfies that $``\Lambda^\Q$ is the unique $\omega_2$-iteration strategy of $\M$". 

We can now show that $\Lambda\rest V_{\k_M}^{M[g]}\in M[g]$ as follows. Given an iteration tree $\T\in V_{\k_M}^{M[g]}$ on $\M$ which is according to $\Lambda$ and has a limit length, we have that $\Lambda(\T)=b$ iff the following holds in $M[g]$: for all $\eta<\k_M$ such that $\T\in H_\eta^{M[g]}$, $\T$ is according to $\Lambda^{\Q_\eta}$ and $b=\Lambda^{\Q_\eta}(\T)$.  Hence, $\Lambda\rest V_{\k_M}^{M[g]}\in M[g]$.
\end{proof}

The following are important corollaries. 

\begin{corollary}\label{fullness of hod} Suppose $\mu$ is good and let $g\subseteq Coll(\omega, \mu)$ be generic. Let $\eta< o(\P_{\mu})$. Then 
\begin{center}
$\W_{\mu, g}(\P_\mu|\eta)=\mathcal{O}^{\P_\mu}_\eta$.
\end{center} 
\end{corollary}
\begin{proof} Suppose first that $\M\insegeq \mathcal{O}^{\P_\mu}_\eta$ is such that $\rho_\omega(\M)=\eta$. To show that $\M\insegeq \W_{\mu, g}(\P_\mu|\eta)$ it is enough to show that $\M$ is $\k^+$-iterable. To show this, it is enough to show the following claim. \\

\textit{Claim.} Suppose $\mu$ is good and let $g\subseteq Coll(\omega, \mu)$ be generic. Let $\eta< o(\P_{\mu})$ and let $\M\inseg \mathcal{O}^{\P_\mu}_\eta$ be such that $\rho_\omega(\M)=\eta$. Then, in $V[g]$, $\M$ is $\kappa^+$-iterable. \\
\begin{proof} Let $(M, \pi)$ be a good hull. It is enough to show the claim in $M$. Let then $\eta< o(\P^M_{\mu})$ and let $\M\inseg \mathcal{O}^{\P^M_\mu}_\eta$ be such that $\rho_\omega(\M)=\eta$. We need to see that $\M$ is $\k_M^+$-iterable in $M$. Notice that since $\P_{\mu}^M$ is $\Gamma_{\mu, g}$-suitable, it follows from clause 1 of \rlem{equivalencies 1} that $\M\insegeq\W^{\Gamma_{\mu, g}}(\P_\mu^M|\eta)$. It then follows from \rlem{correctness of good hulls} that $\M$ is $\k_M^+$-iterable in $M$.
\end{proof}

Let now $\M\insegeq \W_{\mu, g}(\P_\mu|\eta)$ be such that $\rho(\M)=\eta$. We want to show that $\M\insegeq \P_\mu$. Let $(M, \pi)$ be a good hull such that $\M\in rng(\pi)$. Let $\N=\pi^{-1}(\M)$. Notice that $\N$ is $\k^+$-iterable and hence, by \rlem{strategies are determined}, $\S_{\mu, g}\models ``\N$ is $\omega_1$-iterable". Because $\P_{\mu}^M$ is $\Gamma_{\mu, g}$-suitable, we have that $\N\insegeq \P_{\mu}^M$. Using elementarity of $\pi$ we get that $\M\insegeq \P_\mu$.
\end{proof}

\begin{corollary}\label{stability of lower parts} Suppose $\mu<\nu$ are two good points. Suppose $g\subseteq Coll(\omega, \mu)$ is $V$-generic and $h\subseteq Coll(\omega, \nu)$ is $V[g]$-generic. Then for $a\in V_\k[g]$
\begin{center}
$\W_{\mu, g}(a)=\W_{\nu, g*h}(a)$.
\end{center}
\end{corollary}
\begin{proof} 
Using clause 2e of \rlem{extending strategies via suitability} applied in $V[g]$, we have that $\W_{\nu, g*h}(a)\insegeq\W_{\mu, g}(a)$. Fix then $\M\inseg \W_{\mu, g}(a)$ such that $\rho_\omega(\M)=a$ and let $(M, \pi)$ be a good hull at $\mu$ such that $\{\eta, \M\}\in rng(\pi)$. Let $\l=\pi^{-1}(\nu)$ and $\N=\pi^{-1}(\M)$. It is enough to show that whenever $k\subseteq Coll(\omega, \l)$ is $M[g]$-generic, $\N\insegeq (\W_{\l, g*k}(\pi^{-1}(a)))^{M[g][k]}$. Moreover, by genericity, it is enough to show this fact for $k\in V[g]$. Fix then such a $k$. The rest of the proof is then just a word-by-word generalization of the proof of \rthm{correctness of good hulls} with  $V=M[g][k]$. We leave it to the reader. 
\end{proof}

\begin{corollary}\label{stability of lower parts in determinacy model} Suppose $\mu<\nu$ are two good points. Suppose $g\subseteq Coll(\omega, \mu)$ is $V$-generic and $h\subseteq Coll(\omega, \nu)$ is $V[g]$-generic. Then for $a\in V_\k[g]$
\begin{center}
$\W^{\Gamma_{\mu, g}}_{\mu, g}(a)=\W^{\Gamma_{\nu, g*h}}_{\nu, g*h}(a)$.
\end{center}
\end{corollary}
\begin{proof}
This follows immediately from clause 1 of \rlem{equivalencies 1} and \rcor{stability of lower parts}.
\end{proof}

Let $\nu_0$ be the least good point and let $h_0\subseteq Coll(\omega, \nu_0)$. We let $W=V[h_0]$ and for $a\in H_{\k^+}^{V[h_0]}$ we let
\begin{center}
$\W(a)=\W^{(\powerset(\powerset(\k^+)))^W}(a)$.
\end{center}
In what follows, we will often deal with $(\powerset(\powerset(\k^+)))^W$-suitable premice. To make the notation convenient, we drop $(\powerset(\powerset(\k^+)))^W$ from our notation. Thus, in what follows, a suitable premouse is $(\powerset(\powerset(\k^+)))^W$-suitable, short tree is $(\powerset(\powerset(\k^+)))^W$-short and etc.  

\subsection{Excellent points and excellent hulls}

\begin{definition}
 We say $\mu>\nu_0$ is \textit{excellent} if $\mu$ is good and $\mu^{\nu_0}=\mu$.
\end{definition}

\begin{definition}
Suppose $\mu$ is excellent and $(M, \pi)$ is a good hull at $\mu$. We say $M$ is an excellent hull if $M^{\nu_0}\subseteq M$.
\end{definition}

Clearly if $\mu$ is excellent then there is an excellent hull at $\mu$. Also, it is clear that there are at least $\omega$ many excellent points. The next two lemmas go back to Ketchersid's work done in \cite{Ketchersid}. 

\begin{lemma}\label{excellent points} Suppose $\mu$ is excellent. Then  $\cf^V(\Theta^{\mu})\leq \nu_0$.
\end{lemma}
\begin{proof} Suppose not and let $\pi : M\rightarrow V_\zeta$ be good at $\nu_0$ such that $\mu\in ran(\pi)$. Let $\l=\pi^{-1}(\Theta^{\mu})$. Because we are assuming $\cf^V(\Theta^{\mu})>\nu_0$ and because $\card{\l}=\nu_0$, we have that
\begin{center}
 $\nu=_{def}\sup(\pi\rest \l)<\Theta^\mu$.
\end{center} 
Let $\Q=\pi^{-1}(\P_{\mu})$. Let $E$ be the $(\cp(\pi), \nu)$-extender derived from $\pi\rest \Q : \Q\rightarrow \P_{\mu}$. We let $\N=Ult(\Q, E_{\pi}\rest \l)$ and let $j:\Q\rightarrow \N$ be the ultrapower embedding. Because $\d^\Q$ is regular in $\Q$, we have that $j\rest \d^\Q$ is cofinal in $j(\d^\Q)$ and hence, $j(\d^\Q)=\nu$. Also, we have that $\P_{\mu}|\nu=\N|\nu$.\\

\textit{Claim.} $\N\insegeq \P_\mu$.\\
\begin{proof}
Notice that if $(\M, k)\in V$ is such that $k:\M \rightarrow \N$ and $\M$ is countable in $V$ then in $W$ there is $\sigma:\M\rightarrow \Q$. Hence, in $W$, $\M$ is $\k^+$-iterable above $\sigma^{-1}(\d^\Q)$. It then follows that in $W$, $\N$ s countably iterable above $\nu$. On the other hand, since $\nu<\Theta^\mu$, we have that $\P_\mu\models ``\nu$ isn't Woodin". Let then $\N^*\insegeq \P_\mu$ be the least initial segment such that $\mathcal{J}_1(\N^*)\models ``\nu$ isn't Woodin". Then by \rcor{fullness of hod}, it follows that $\N^*$ is countably iterable. Therefore, because $\N\models ``\nu$ is Woodin", we have that $\N\insegeq \N^*\insegeq \P_\mu$.
\end{proof}

Let $\S$ be the largest initial segment of $\P_{\mu}$ such that $\S\models ``\nu$ is Woodin". Because $\S\in \H^{S_{\mu, h_0*h}}$ for any $W$-generic $h\subseteq Coll(\omega, \mu)$, it follows from the fact that MC holds in $\S_{\mu, h_0*h}$ that
\begin{center}
$\powerset(\d^\Q)\cap L[\S, \Q]=(\powerset(\d^\Q))^\Q$.
\end{center}
It then also follows that we can lift $j$ to 
\begin{center}
$j^+:L[\S, \Q]\rightarrow L[j^+(\S), \N]$.
\end{center}
Because $E$ is countably closed, $j^+(\S)$ is countably iterable in $V$. This means that $\S\in L[j^+(\S), \N]$ as it can be identified in $L[j^+(\S), \N]$ as the unique sound $\N$-premouse $\M$ such that $\rho_\omega(\M)=\nu$, $\mathcal{J}_1(\M)\models ``\nu$ is not Woodin" and in $L[j^+(\S), \N]^{Coll(\omega, \N)}$, there is $\sigma: \M\rightarrow j^+(\S)$. However, $L[j^+(\S), \N]\models ``j(\d^\Q)=\nu$ is Woodin", contradiction!
\end{proof}

The proof of the lemma also gives the following the proof of which we leave to our readers.

\begin{lemma}\label{cofinality of theta} Suppose $\mu$ is good. Then $\cf^V(\Theta^\mu)\leq \mu$.
\end{lemma}


%
%

Let $T_{\mu}\in V$ be the tree on $\omega\times (\d^2_1)^{\S_{\mu, g}}$ such that whenever $g\subseteq Coll(\omega, \mu)$ is generic, 
\begin{center}
$S_{\mu, g}\models`` p[T_\mu]$ is the universal $\Sigma^2_1$-set"
\end{center}
where $p[T]$ is the projection of $T$. Next we show that realizable premice are suitable. 


\begin{lemma}\label{realizability gives fullness lemma}
Suppose $\mu$ is an excellent point, $(M, \pi)$ is an excellent hull at $\mu$ and $g\subseteq Coll(\omega, \mu)$ is $W$-generic. Let $H\in (V_{\k_M})^{M[h_0][g]}$ be such that $\W(H)\models`` H=H_\nu$ for some cardinal $\nu$". Suppose further that $N\in V_{\k_M}^{M[h_0][g]}$ is such that there are embeddings $\sigma: \W(H)\rightarrow N$ and $\tau: N\rightarrow \pi(\W(H))$ such that $\pi\rest \W(H)=\tau\circ \sigma$. Then 
\begin{center}
$N=\W(\sigma(H))$.
\end{center}
\end{lemma}
\begin{proof} Let $\eta>\mu$ be a good point such that for some $\l$, $\pi(\l)=\eta$ and $H, N\in V_{\l}^{M[h_0][g]}$. Let $T\in M$ be such that $\pi(T)=T_\eta$. Using \rcor{stability of lower parts in determinacy model}, we have that $L[T, \W(H)]\models ``H=H_\nu$ for some cardinal $\nu$". This implies that we can extend $\sigma$ to act on $L[T, \W(H)]$.
We let $\sigma^+:L[T, \W(H)]\rightarrow L[\sigma^+(T), N]$ be this extension of $\sigma$. It then follows that we can find $\tau^+:L[\sigma^+(T), N]\rightarrow L[\T_\eta, \pi(\W(H))]$ extending $\tau$ and such that 
\begin{center}
$\pi\rest L[T, \W(H)] = \tau^+\circ \sigma^+$. 
\end{center}

Fix now an $n\in \omega$ such that whenever $h\subseteq Coll(\omega, \eta)$ is $W$-generic then in $\S_{\eta, h}$, $(T_\eta)_n$ projects to the set $\{ (x, y)\in \bR^2: x$ codes a set $a$ and an $a$-mouse $\M_x$ and $y$ codes an $a$-mouse $\M_y$ such that $\rho_\omega(\M_y)=a$ and $\M_x\inseg \M_y\}$. Notice that the following holds in $L[T, \W(H)]$: for any generic $k\subseteq Coll(\omega, \W(H))$, in $L[T, \W(H)][k]$, for any $x\in \bR$ coding $(H, \W(H))$ and for any real $y$, $(x, y)\not \in p[T_n]$. We let this sentence be $\phi$ (in the language of $L[T, \W(H)]$). 

We have that $\phi$ holds in $L[\sigma^+(T), N]$. Since $\tau: N\rightarrow \pi(\W(H))$, we also have that $N\insegeq \W(\sigma(H))$. Moreover, since $\powerset(\W(H))\cap L[T, \W(H)]\subseteq \W(\W(H))$, we get that, by the same argument as above, that $\powerset(N)\cap L[\sigma^+(T), N]\subseteq\W(N)$. It then follows that $\powerset(N)\cap L[\sigma^+(T), N]$ is countable in $W[g]$.

Suppose now that $N\not =\W(\sigma(H))$. Because we already know that $N\insegeq \W(\sigma(H))$, it is enough to show that  $\W(\sigma(H))\insegeq N$. Suppose not. There is then a $\k^+$-iterable sound $\M\insegeq \W(\sigma(H))$ such that $\rho_\omega(\M)=\sigma(H)$ and $N\inseg \M$. Fix $k\subseteq Coll(\omega, \sigma(H))$ such that $k\in W[g]$ and $k$ is both $M[h_0][g]$ and $L[\sigma^+(T), N]$ generic. Let $x\in (\bR)^{L[\sigma^+(T), N][k]}$ code $(\sigma(H), N)$ in a canonical fashion. We have that $x\in M[h_0][g][k]$. Let $y\in \bR^{M[g][k]}$ code $\M$. We have that $(x, y)\in p[T_n]$. It then follows that $(x, y)\in p[\sigma^+(T_n)]$. Thus, by absoluteness, we can find $w\in L[\sigma^+(T), N][k]$ such that $(x, w)\in p[\sigma^+(T_n)]$. This, however, contradicts $\phi$. 
\end{proof}

The following is an immediate corollary of \rlem{realizability gives fullness lemma}. We will use it to produce strategies which are fullness preserving. 

\begin{corollary}\label{realizability gives fullness} Suppose $\mu$ is an excellent point, $(M, \pi)$ is a good hull at $\mu$ and $g\subseteq Coll(\omega, \mu)$ is $W$-generic. Suppose $a\in H_{\omega_1}^{M[h_0]}$ and suppose further that $\Q\in H_{\kappa_M}^{M[h_0][g]}$ is an $a$-premouse such that there are $\Sigma_1$-elementary embeddings $\sigma:\P^M_{\mu, a} \rightarrow \Q$ and $\tau:\Q\rightarrow \P_{\mu, a}$ such that
\begin{center}
$\pi\rest \P^M_{\mu, a}= \tau\circ \sigma$.
\end{center}
Then $\Q$ is suitable.
\end{corollary}

\subsection{A quasi-sjs}

Our ultimate goal is to isolate a quasi-self justifying system. Excellent hulls give possible candidates of such quasi-self justifying systems. Suppose $\mu$ is excellent and $g\subseteq Coll(\omega, \mu)$ is $W$-generic. Given an excellent $\mu$, we let $F^\mu=F_{od}^{\S_{\mu, h_0^*g}}$ (we drop the generics to save space). Let $(M, \pi)$ be an excellent hull at $\mu$. We then let $F^{M, \mu}=\pi^{-1}(F^\mu)$ and $F^{M, \pi}=\pi"F^{M, \mu}$.  We will eventually show that $F^{M, \pi}$ is a quasi-sjs. The proof will be spread over several lemmas. Our first lemma is the following.

 \begin{lemma}\label{fullness} Suppose $\mu$ is excellent, $g\subseteq Coll(\omega, \mu)$ is $W$-generic, $(M, \pi)$ is an excellent hull at $\mu$ and $a\in (H_{\omega_1})^{M[g]}$. Then 
 \begin{center}
 $\P^M_{\mu, a}=\cup_{f\in F^{M, \pi}}H_f^{\P_{\mu, a}}$. 
 \end{center}
 \end{lemma}
 \begin{proof} This is an immediate consequence of the definition of $\M_\infty$. 
  \end{proof}
 
\begin{lemma}\label{quasi iterability of hod} Suppose $\mu$ is excellent, $g\subseteq Coll(\omega, \mu)$ is $W$-generic, $(M, \pi)$ is an excellent hull at $\mu$ and $a\in (H_{\omega_1})^{M[g]}$. Then in $\S_{\mu, h_0*g}$, $\P^M_{\mu, a}$ is $(F^\mu, F^{M, \pi})$-quasi iterable.
\end{lemma}
\begin{proof} Using \rthm{existence of quasi-iterable premice} and elementarily of $\pi$, we get that for every $f\in F^{M, \pi}$ there is $\Q\in M[h_0][g]$ such that $\S_{\mu, h_0*g}\models  ``\Q$ is $(F^\mu, f)$-quasi iterable". It then follows that $\P^M_{\mu, a}$ is $(F^\mu, f)$-quasi iterate of $\Q$. Since $f$ was arbitrary, we get that in $\S_{\mu, h_0*g}$, $\P^M_{\mu, a}$ is $(F^\mu, f)$-quasi iterable for every $f\in F^{M, \pi}$. It then follows that $\S_{\mu, h_0*g}\models ``\P_{\mu, a}$ is $(F^\mu, F^{M, \pi})$-quasi iterable". 
\end{proof}

If we restrict ourselves to stacks that are in $V_{\k_M}^M$ then we actually get that $F^{M, \pi}$ is a qsjs as witnessed by $\P_{\mu, a}$.

\begin{lemma}[$F^{M, \pi}$ is a qsjs]\label{getting a qsjs}
Suppose $\mu$ is excellent, $g\subseteq Coll(\omega, \mu)$ is $W$-generic, $(M, \pi)$ is an excellent hull at $\mu$ and $a\in (H_{\omega_1})^{M[g]}$. Then in $\S_{\mu, h_0*g}$, $F^{M, \pi}$ is a qsjs as witnessed by $\P_{\mu, a}^M$ for stacks in $V_{\k_M}^{M[h_0][g]}$. 
\end{lemma}
\begin{proof}
We do the proof for $a=\emptyset$. Clearly $F^{M, \pi}$ and $\P=\P_\mu^M$ satisfy the first two clauses of \rdef{qsjs}. We check clause 3 for quasi iterations that are in $V_{\k_M}^{M[h_0][g]}$. Let $\VT=\la \T_\a, \M_\a: \a<\nu\ra\in V_{\k_M}^{M[h_0][g]}$ be an $(F^{M, \pi}, F^{M, \pi})$-quasi iteration of $\P$. We need to show that it has a last model. Let $\la \pi^{g, \a}_{\gg, \xi} : \gg<\xi<\a\leq \nu\ra$ be the embeddings of $\VT$. We need to show the $\VT$ has a last model and for this we need to consider two cases. Suppose first $\nu=\a+1$. Let $\Q$ be the last model of $\T_\a$. We need to show that $\Q=\cup_{f\in F^{M, \pi}}H^\Q_f$. \\

\textit{Claim.} $\Q=\cup_{f\in F^{M, \pi}}H^\Q_f$.\\
\begin{proof} Since $\Q$ is a $(F^{M, \pi}, F^{M, \pi})$-quasi iterate of $\P$, it is also $(F^\mu, F^{M,\pi})$-quasi iterate of $\P$. Hence, in $\S_{\mu, h_0*g}$, $\Q$ is $F^{M, \pi}$-iterable. Let  $\S$ be the transitive collapse of $\cup_{f\in F^{M, \pi}}H^\Q_f$ and let $\sigma: \S\rightarrow \Q$ be the inverse of the transitive collapse. Let $l=\pi_{\Q, \infty, F^{M, \pi}}$. Then $l\circ \sigma: \S\rightarrow \pi(\P)$. Because $\P=\cup_{f\in F^{M, \pi}}H^\P_f$, we have $k:\P\rightarrow \S$ such that
\begin{center}
$\pi\rest P= l\circ \sigma \circ k$.
\end{center}
Because $\S\in M$ (this follows from the fact that $\S\insegeq \Q$), it follows from \rlem{realizability gives fullness lemma} that $\S$ is suitable and hence, $\S=\Q$. It then follows that $\Q=\cup_{f\in F^{M, \pi}}H^\Q_f$.
\end{proof}

Next we assume $\nu$ is limit. In this case we have to verify that if for $f\in F^{M, \pi}$, $H_f$ is the direct limit of $H_f^{\M_\gg}$'s under $\pi^{f, \nu}_{\gg, \xi}$'s then $\Q=\cup_{f\in F^{M, \pi}}H_f$ is suitable. The proof is similar to the proof of the claim. We define $l:\Q\rightarrow \pi(\P)$ as follows. Given $x\in \Q$ let $\a, f$ be such that $\a\in [\b_{f, \nu}, \nu)$ and there is $\bar{x}\in H^{\M_\a}_f$ such that $\pi^{f, \nu}_{\a, \nu}(\bar{x})=x$. Then let $l(x)=\pi_{\M_\a, \infty, f}(\bar{x})$. We then have $k:\P\rightarrow \Q$ coming from the direct limit construction and such that $\pi\rest \P=l\circ k$. Therefore, using  \rlem{realizability gives fullness lemma}, we get that $\Q$ is suitable.

Actually, a little argument is needed to show that $\Q\in M$. To see this, first let $\vec{B}=\la B_i: i<\omega\ra \subseteq F^{M, \pi}$ be such that $\sup_{i<\omega}w(B_i)=\Theta^\mu$. Let $D_i=\{ (x, y) : x\in \bR^{M[h_0*g]}, x$ codes a continuous function $f$ and $y\in f^{-1}" B_i\}$. Then for any $\eta<\k_M$ and any $M[h_0*g]$-generic $k\subseteq Coll(\omega, \eta)$ such that $k\in W[g]$, $\la D_i\cap M[h_0*g*k] : i<\omega\ra\in M[h_0*g*k]$. It is then easy to see that this is enough to compute $\Q$ in $M[h_0*g]$.

As part of the proof of clause 3 we have also shown clause 4. We leave the details to the reader. To show clause 5, let $\Q\in V_{\k_M}^{M[h_0][g]}$ be a $(F^{M, \pi}, F^{M, \pi})$-quasi iterate of $\P$ and $\sigma:\R\rightarrow_{\Sigma_1} \Q$ be such that $f(\Q)\in rng(\sigma)$ for all $f\in F^{M, \pi}$ and $\R\in M$. We need to see that $\R$ is suitable. Again, we have $rng(\pi_{\P, \Q, F^{M, \pi}})\subseteq rng(\sigma)$ and hence, there is an embedding $k:\P\rightarrow_{\Sigma_1} \R$ such that $\pi_{\P, \Q, F^{M, \pi}}=\sigma\circ k$. Let then $l=\pi_{\Q, \infty, F^{M, \pi}}\circ \sigma$. It follows that $\pi\rest \P=l\circ k$ and hence, $\R$ is suitable. 
\end{proof}

Next, we show that we can substitute $F^{M, \pi}$ by a subset of it which is countable in $W$. We make this move as we believe it makes the arguments that follow more transparent.  We could just as well carry on with $F^{M, \pi}$. 

\begin{definition}[Pre-sjs] Suppose $\mu$ is excellent and $g\subseteq Coll(\omega, \mu)$ is $W$-generic. We say $\vec{B}=\la B_i: i<\omega\ra$ is a pre-sjs at $\mu$ if for every $a\in H_{\omega_1}^{W[g]}$.
\begin{enumerate}
\item for every $i<\omega$, $B_i\in F^\mu$,
\item for every $n<\omega$ there is $k<\omega$ such that $B_k=\oplus_{i<n}B_i$,
\item $\P_{\mu, a}=\cup_{i<\omega}H^{\P_{\mu, a}}_{B_i}$.
\end{enumerate}
\end{definition}

\begin{lemma}\label{getting pre-sjs} Suppose $\mu$ is excellent and $g\subseteq Coll(\omega, \mu)$ is $W$-generic. Then there is a pre-sjs at $\mu$.  
\end{lemma}
\begin{proof} It is enough to prove that for some excellent $(M, \pi)$, $M\models ``$there is a pre-sjs at $\mu$". Fix then some excellent $(M, \pi)$. We have that $M[h_0]\models \cf(\d^{\P^M_{\mu}})=\omega$. Fix then some increasing sequence $\la \gg_i: i<\omega\ra\in M$ cofinal in $\d^{\P^M_{\mu}}$. Let $\d=\d^{\P^M_\mu}$. We can then find $\vec{B}^*=\la B^*_i: i<\omega\ra\subseteq F^{M, \pi}$ such that for every $a\in H_{\omega_1}^{M[h_0*g]}$, $\sup_{i<\omega}\gg_{B_i}^{\P_{\mu, a}^M}=\d^{\P_{\mu, a}^M}$. Next, close $\vec{B}^*$ under $\oplus$ operator to obtain $\vec{B}=\la B_i: i<\omega\ra$. We claim that $M\models``\pi^{-1}(\vec{B})$ is a pre-sjs". 

To see this, it is enough to show that for every $a$, $\P_{\mu, a}^M=\cup_{i<\omega}H_{B_i}^{\P^M_{\mu, a}}$. Let then $\T$ on $\P_{\mu, a}^M$ be the correctly guided maximal tree for making $\P_{\mu, a}^M|(\d^{+n-1})^{\P^M_{\mu, a}}$ generic where we let $\d^{+-1}=\d$. We have $\T\in \P_{\mu, a}^M$ and $\d(\T)=(\d^{+n})^{\P^M_{\mu, a}}$. Also, because $\P_{\mu, a}^M$ is $(F^{M,\pi}, F^{M, \pi})$-quasi iterable, we have a branch $b$ of $\T$ such that $\M^\T_b$ is suitable and $i^\T_b(\d)=\d(\T)$. Moreover, letting $\Q=\M^\T_b$, for every $k, n<\omega$, $i^\T_b(\tau_{B_k, n}^{\P_{\mu, a}^M})=\tau_{B_k, n}^{\Q}$. Since for each $i<\omega$ there is $k<\omega$ such that $\pi_{\P^M_{\mu, a}, Q, B_i}\in H_{B_k}^{\P_{\mu, a}^M}$, we have that if $\S$ is the transitive collapse of $\cup_{i<\omega}H^{\P^M_{\mu, a}}_{B_i}$ and $\sigma:\S\rightarrow \P^M_{\mu, a}$ is the inverse of the collapse then $\sigma\rest (\d^{+n})^\S$ is cofinal implying that $\sigma=id$.
\end{proof}

Suppose $\mu$ is excellent, $(M, \pi)$ is an excellent hull at $\mu$, $g\subseteq Coll(\omega, \mu)$ is $V$-generic and $\vec{B}\in V[g]$ is a pre-sjs at $\mu$. We let $B_i^M=_{def}B_i\cap M[g]$ and $\vec{B}^M=_{def}\la B_i^M:i<\omega\ra$. We say $(M, \pi)$ \textit{captures} $\vec{B}$ if  $\vec{B}^M\in M$ and $\pi(\vec{B}^M)=\vec{B}$. We let $F_{\vec{B}}=\{ f_{B_i}: i<\omega\}$. It follows from \rlem{getting a qsjs} and the proof of \rlem{getting pre-sjs} that

\begin{lemma}\label{getting qsjs} Suppose $\mu$ is excellent, $g\subseteq Coll(\omega, \mu)$ is $W$-generic, $(M, \pi)$ is an excellent hull at $\mu$ and $a\in (H_{\omega_1})^{M[g]}$. Then in $\S_{\mu, h_0*g}$, $F_{\vec{B}}$ is a qsjs as witnessed by $\P_{\mu, a}^M$ for stacks in $V_{\k_M}^{M[h_0][g]}$. 
\end{lemma}

We continue with the set up of the lemma. Given any $B\in \powerset(\bR^{W[g]})\cap OD^{\S_{\mu, h_0*g}}$, we let
\begin{center}
$f^{M,g}_B=\{ (\Q, f_B(\Q)): \Q\in V_{\k_M}^{M[g]}$ and $\Q$ is $\Gamma_{\mu, g}$-suitable$\}$.
\end{center}
We define $f^{M}_B$ similarly. Given a pre-sjs $\vec{B}$ at $\mu$ such that $M$ captures $\vec{B}$, we let 
\begin{center}
$f^{M, g}_{\vec{B}}=\la f^{M, g}_{B_i}: i<\omega\ra$ and $f^M_{\vec{B}}=\la f^M_{B_i}: i<\omega\ra$.
\end{center}

\begin{lemma}\label{f is in M} Suppose $\mu$ is excellent, $(M, \pi)$ is an excellent hull at $\mu$, $g\subseteq Coll(\omega, \mu)$ is $W$-generic and $\vec{B}\in V[g]$ is a pre-sjs at $\mu$ captured by $(M, \pi)$. Then $f^{M, g}_{\vec{B}}\in M[h_0*g]$.
\end{lemma} 
\begin{proof}
We have that in $\S_{\mu, h_0*g}$, for each $i$, $f^{M,g}_{B_i}$ is $OD_{V_{\k_M}^{M[h_0][g]}}$ and hence, $f^{M, g}_{B_i}\in M[h_0][g]$ (recall that $\W(A_M)\in M$). Let then for each $i<\omega$, $\dot{f}_i\in M[h_0]^{Coll(\omega, \mu)}$ be the term that is always realized as $f^{M, g}_{B_i}$. Because $M[h_0]$ is $\omega$-closed in $W$, we have that $\la \dot{f}_i: i<\omega\ra\in M[h_0]$. Hence, $f^{M, g}_{\vec{B}}\in M$.
\end{proof}

Suppose now $\mu$ is excellent, $g\subseteq Coll(\omega, \mu)$ is $W$-generic, $\vec{B}$ is a pre-sjs at $\mu$, and $(M, \pi)$ is an excellent hull at $\mu$ capturing $\vec{B}$. Suppose further that $a\in (H_{\omega_1})^{M[g]}$. \rlem{getting a qsjs} and \rlem{getting strategies from qsjs} imply that we can define $\Sigma^{F_{\vec{B}}}$ for $\P_{\mu, a}$. In our case, however, we can only restrict to iterations that are in $V_{\k_M}^{M[h_0][g]}$. We then let $\Sigma^{M, g}_{\mu, \vec{B}, a}$ be this iteration strategy. 

\begin{lemma} $\Sigma^{M, g}_{\mu, \vec{B}, a}$ is a $(\k_M, \k_M)$-iteration strategy which acts on iterations that are in $M[h_0][g]$ and is fullness preserving. Moreover, if $\VT\in V_{\k_M}^{M[h_0][g]}$ is via $\Sigma^{M, g}_{\mu, \vec{B}, a}$ and the last normal component of $\VT$ has limit length then $\Sigma^{M, g}_{\mu, \vec{B}, a}(\VT)\in M[h_0][g]$.
\end{lemma} 
\begin{proof} The first part is an immediate consequence of \rlem{getting strategies from qsjs}. The second part follows from the fact that $f^{M, g}_{\vec{B}}\in M[h_0][g]$.
\end{proof}

Notice that $\Sigma^{M, g}_{\mu, \vec{B}, a}\rest V_{\k_M}^{M[h_0]}\in M[h_0]$ given that $a\in M[h_0]$. We let $\Sigma^{M}_{\mu, \vec{B}, a}=\Sigma^{M, g}_{\mu, \vec{B}, a}\rest V_{\k_M}^{M[h_0]}$. When $a=\emptyset$, we omit it from subscripts. In the next subsection we will show that for each $i<\omega$ there is some $\Sigma^{M, g}_{\mu, \vec{B}, a}$-iterate $\Q$ of $\P_{\mu, a}^M$ such that if $\Lambda$ is the strategy of $\Q$ induced by $\Sigma^{M, g}_{\mu, \vec{B}, a}$ then $\Lambda$ strongly respect $B_i$. It then easily follows from \rlem{getting branch condensation} that $\Lambda$ has branch condensation.

\subsection{An $\omega$-suitable $\P$}

Recall that we have that $\cf^W(\k)=\omega$. Let then $\la \mu_i: i<\omega\ra$ be a sequence of excellent points such that for each $i$, $\mu_{i+1}^{\mu^+_i}=\mu_{i+1}$ and $\sup_{i<\omega}\mu_i=\k$. For the rest of the paper, we fix a $W$-generic $g_0\subseteq Coll(\omega, \mu_0)$ and a pre-sjs $\vec{B}$ at $\mu_0$. 
%

Fix now an excellent $(M^*, \pi^*)$ at $\mu_0$ capturing $\vec{B}$. For each $X\in V_{\k_{M^*}}^{M^*[h_0][g_0]}$ we let $\T_X$ be the tree on $\P^{M^*}_{\mu}$ according to $\Sigma^{M^*}_{\mu_0, \vec{B}}$ that makes $X$ generic. Let $\Q_X$ be the last model of $\T_X$ and let $\pi_X:\P^{M^*}_{\mu}:\rightarrow \Q_X$ be the iteration embedding according to $\Sigma^{M^*}_{\mu_0, \vec{B}}$. We then let $\tau_{B_i, X}=\pi_X(\tau_{B_i, 0}^{\P_{\mu}^{M^*}})$.
We let 
\begin{center}
$\tau_i=\{ (X, \Q_X, \tau_{B_i, X}) : X\in V_{\k_{M^*}}^{M^*[h_0][g]}\}$ 
\end{center}
and let $\vec{\tau}=\la \tau_i: i<\omega\ra$. The proof of \rlem{f is in M} gives that $\vec{\tau}\in M^*[h_0][g]$. 

To continue, we introduce the following notation. Suppose $\mu\leq \nu$ are two excellent points and $(M, \pi)$ and $(N, \sigma)$ are two excellent hulls at $\mu$ and $\nu$ respectively. We write $(M, \pi)<(N, \sigma)$ if $(M, \pi)\in rng(\sigma)$. In this case, we let $\pi_{M, N}=\sigma^{-1}(\pi): M\rightarrow N$. We have that $\pi_{M, N}\in N$.

 Our first goal here is to show that we can extend $\vec{B}$ to a pre-sjs at any other excellent point $>\mu_0$.
 Given an excellent $\mu>\mu_0$, we let $B_{\mu, i}\in W[g_0]^{Coll(\omega, \mu)}$ be the name for the set of reals given by $\pi^*(\tau_i)$ as follows: given a standard name for a real $\dot{x}\in W[g_0]^{Coll(\omega, \mu)}$ and $p\in Coll(\omega, \mu)$, 
\begin{center}
$(p, \dot{x})\in B_{\mu, i}\iff W[g_0]\models``$ if $\eta=\mu^{++}$ and $(\Q, \tau)$ is such that $(V_\eta, \Q, \tau)\in \pi^*(\tau_i)$ then $p\forces^{\Q[V_\eta]}_{Coll(\omega, \mu)}\dot{x}\in \tau$".
\end{center}
If $g\subseteq Coll(\omega, \mu)$ is $W[g_0]$ generic then we let $B_{\mu, i, g}$ be the realization of $B_{\mu, i}$. We let $\vec{B}_\mu=\la B_{\mu, i}: i<\omega\ra$ and $\vec{B}_{\mu, g}=\la B_{\mu, i, g} : i<\omega\ra$.

\begin{lemma}\label{Bs are determined} Suppose $\mu>\mu_0$ and $g\subseteq Coll(\omega, \mu)$ is $W[g_0]$-generic. Then $B_{\mu, i, g}\in \S_{\mu, h_0*g_0*g}$.
\end{lemma}
\begin{proof} It is enough to show that $L(B_{\mu, i, g}, \bR^{W[g_0*g]})\models AD^+$. Let $(M, \pi)$ be a good hull at $\mu_0$ such that $(M^*, \pi^*)<(M, \pi)$ and $\mu \in rng(\pi)$. Let $B=\pi^{-1}(B_{\mu, i})$. It is enough to show that if $\l=\pi^{-1}(\mu)$ and $h\subseteq Coll(\omega, \l)$ is $M[h_0*g_0]$-generic such that $h\in W[g_0]$ then $M[h_0*g_0*h]\models L(B_{h}, \bR)\models AD^+$ where $B_h$ is the realization of $B$ in $M[h_0*g_0*h]$. Fix then such an $h$. Notice that we have that
\begin{center}
$B_{h}=B_i\cap \bR^{M[h_0*g_0*h]}$.
\end{center}
Since $B_i$ is $OD$ in $\S_{\mu_0, h_0*g_0}$, we have that there is a sound $\bR^{M[h_0*g_0*h]}$-mouse $\M$ such that $B_h\in \M$, $\rho_\omega(\M)=\bR^{M[h_0*g_0*h]}$ and $\S_{\mu_0, h_0*g_0}\models ``\M$ is $\omega_1$-iterable". It then follows from \rlem{equivalencies 1} and \rlem{correctness of good hulls} that $\M\in M[h_0*g_0*h]$ and $M[h_0*g_0*h]\models ``\M$ is $\k_M^+$-iterable". Therefore, $\M\insegeq (\S^-_{\l, h_0*g_0*h})^{M[h_0*g_0*h]}$ implying that $M[h_0*g_0*h]\models L(B_{h}, \bR)\models AD^+$.
\end{proof}

\begin{lemma} Suppose $\mu>\mu_0$ and $g\subseteq Coll(\omega, \mu)$ is $W[g_0]$-generic. Then $\vec{B}_{\mu, i, g}$ is a pre-sjs at $\mu$.  
\end{lemma}
\begin{proof} The proof is like the proof of \rlem{Bs are determined}. It follows from the proof of \rlem{getting pre-sjs} that it is enough to show that for any $a\in H_{\omega_1}^{W[g_0*g]}$, $\sup_{i<\omega} \gg^{\P_{\mu, a}}_{B_{\mu, i, g}}=\Theta^\mu$. 
We first show this for $a=\emptyset$. Let $(M, \pi)$ be an excellent hull at $\mu_0$ such that $(M^*, \pi^*)<(M, \pi)$ and $\mu \in rng(\pi)$. Let $\l=\pi^{-1}(\mu)$ and let $h\subseteq Coll(\omega, \l)$ be $M[h_0*g_0]$-generic such that $h\in W[g_0]$. It is enough to show that in $M[h_0*g_0*h]$, 
\begin{center}
$\sup_{i<\omega} \gg^{\P}_{B_{\mu, i, g}}=\d^{\P}$ 
\end{center} 
where $\P=\pi^{-1}(\P_{\mu})$. This follows from the fact that $(B_{\mu, i, g})^{M[h_0*g_0*h]}=B_i\cap M[h_0*g_0*h]$ and in $M$, $\P$ is a $F_{\vec{B}}$-quasi-iterate of $\P^M_{\mu_0}$.

To show the claim for every $a$ use \rlem{getting a qsjs} and \rlem{getting pre-sjs}. 
\end{proof}

We define $\la \P_i: i<\omega\ra$ as follows:
\begin{enumerate}
\item $\P_0=\P_{\mu_0}$,
\item $\P_{i+1}=\P_{\mu_{i+1}, \P_i}$.
\end{enumerate}
Let $\d_i=\d^{\P_i}$. 

\begin{lemma} $\P_{i+1}\models ``\d_i$ is Woodin" and no level of $\P_{i+1}$ projects across $\d_i$.
\end{lemma}
\begin{proof}
It is enough to show that no level of $\P_{i+1}$ projects across $\d_i$. Towards a contradiction, suppose $\M\insegeq\P_{i+1}$ is such that $\rho(\M)\leq \d_i$. It follows from \rcor{fullness of hod} that $\M$, regarded as a mouse over $\P_i$, is $\k^+$-iterable in $W$ and hence, $\M\insegeq \P_i$, contradiction. 
\end{proof}

Let $\d_\omega=\sup_{i<\omega}\d_i$ and $\P^-=\cup_{i<\omega}\P_i$ and let 
 \begin{center}
   $\P = 
     \begin{cases}
      \W(\P^-)  &:\ \text{if no level of}\ \W(\P^-) \ \text{projects across}\ \d_\omega \\
      \N &:\ \text{where}\ \N\inseg \W(\P^-)\ \text{is the least such that} \ \rho(\N)<\d_\omega\\
     \end{cases}$
\end{center}
Notice that $\k=\d_\omega$. In subsequent sections we will show that $\P=\W(\P^-)$. Before we move on, we fix some notation.

Suppose $(M, \pi)$ is an excellent hull at $\mu_k$ for some $k\leq \omega$. Let $\mu_{-1}=\nu_0$. We say $(M, \pi)$ is \textit{perfect} if $(M^*, \pi^*)<(M, \pi)$, $M^{\mu_{k-1}^+}\subseteq M$ and $\la \mu_i: i<\omega\ra\in rng(\pi)$. Notice that we have that $\P\in rng(\pi)$. We then let $\la \mu_i^M: i\in \omega\ra=\pi^{-1}(\la \mu_i: i\in \omega\ra)$, $\P_i^M=\pi^{-1}(\P_i)$ and $\P^M=\pi^{-1}(\P)$. 

Notice that if for some $k<\omega$, $g\subseteq Coll(\omega, \mu_k)$ is $W[g_0]$ generic and $(M, \pi)$ is a perfect hull at $\mu_k$ then $\Sigma^{M, g_0*g}_{\mu_k, a, \vec{B}_{\mu_k, g}}$ is defined for every $a\in (H_{\omega_1})^{M[h_0*g_0*g]}$. We then let $\Sigma^{M, g}_{\mu_k, a}=\Sigma^{M, g_0*g}_{\mu_k, a, \vec{B}_{\mu_k, g}}$ where we let $\P_{-1}=\emptyset$. Notice that we have that $\Sigma^{M, g}_{\mu_k, a}\rest M[h_0*g_0]\in M[h_0*g_0]$. We then let $\Sigma^M_{\mu_k, a}=\Sigma^{M, g}_{\mu_k, a}\rest  M[h_0*g_0]$. It is again true that $\Sigma^M_{\mu_k, a}\in M[h_0*g_0]$.

\subsection{Iteration strategy for $\P$}

Suppose $(R, \tau)> (M^*, \pi^*)$ is a perfect hull at some $\mu_p$ for $p\leq \omega$. Here we will describe a $(\k_R, \k_R)$-strategy for $\P^{R}$ which acts on stacks all of whose normal components are in $V_{\k_R}^R$ and are above $\d_{p-1}^{\P^R}$ (recall that $\d_{-1}=0$). To do this, we introduce the auxiliary game $\mathcal{G}^a(\P^R, \omega^3, \k_R)$ in which player $II$, after a notification that $I$ is about to start a new round, will play an embedding $\pi_\a: \M_\a\rightarrow \P$ such that $\tau\rest \P=\pi_\a\circ i_{0, \a}$ where $i_{0, \a}$ is the iteration embedding. More precisely, $\la \T_\a, \M_\a, \pi_\a: \a<\eta\ra\subseteq V_{\k_R}^R$ is a run of $\mathcal{G}^a(\P^R, \omega^3, \k_R)$ in which neither player has lost if 
\begin{enumerate}
\item for $\a<\eta$, $\M_\a$ is the model at the beginning of the $\a$th round and $\M_0=\P^R$,
\item for $\a<\eta$, $\M_\a$ is (anomalous) $\omega$-suitable premouse,
\item for $\a<\eta$, $\T_\a$ is a tree with no fatal drops based on some window $(\d_k^{\M_\a}, \d_{k+1}^{M_\a})$ for $k\geq p-1$, 
\item for $\a<\eta$, $\T_\a$ is correctly guided and if $\T_\a$ has a last branch and $\T_\a^-$ is short then $\Q(\T_\a^-)\insegeq \M^{\T_\a}_{lh(\T_\a)-1}$,
\item $\pi_0=\pi\rest \P^R$ and for $\a<\b<\eta$, $\pi_\a:\M_\a\rightarrow \P$,
\item if for $\a<\b< \eta$, $i_{\a, \b}:\M_\a\rightarrow \M_\b$ is the iteration embedding then for $\a<\b<\eta$, $\pi_{\a}=\pi_\b\circ i_{\a, \b}$.
\end{enumerate}
The game has at most $\omega^3$ many rounds. As usual player $I$ plays the extenders and starts new rounds. $II$ plays branches. $I$ cannot start a new round if the current model isn't suitable. If $I$ decides to start a new round with $\M_\a$ as its starting model then $II$ has to play the embedding $\pi_\a$ that satisfies clauses 5 and 6. The game stops if one of the models produced is ill founded or if $II$ cannot play $\pi_\a$. $II$ wins if all the models appearing in the iteration are wellfounded, she can always play the embeddings $\pi_\a$, the game runs $\omega^3$ many rounds or if one of the rounds runs $\k_R$-steps. 

\begin{lemma} Player $II$ has a wining strategy in $\mathcal{G}^a(\P^R, \omega^3, \k_R)$.
\end{lemma}
\begin{proof} We define the strategy $\Lambda$ for $II$ by induction. First fix some bijection $f:\k_R\rightarrow V_{\k_R}^R$ such that $f\in R$. In the first round, $I$ plays a tree on $\P^R=\M_0$. Suppose $I$ starts playing the first round on $(\d_k^{\M_0}, \d_{k+1}^{\M_0})$ for $k\geq p-1$ and suppose $\T$ is a tree constructed according to $\Lambda$ and $\T$ has limit length. Below we describe how $II$ should play her move. 
\begin{enumerate}
\item If $\T$ is short then $II$ plays the unique cofinal branch $b\in R$ of $\T$ such that $\Q(b, \T)\insegeq \W(\M(\T))$ exists,
\item if $\T$ is maximal then $II$ plays $b$ such that
\begin{enumerate}
\item $b\in R$,
\item $\M^\T_b$ is $\omega$-suitable,
\item there is $\sigma:\M^\T_b\rightarrow \P$ such that $\tau\rest \P^R=\sigma\circ i^\T_b$.
\item $b$ is the $f$-least branch satisfying 2a-2c.
\end{enumerate}
\end{enumerate}
We need to see that $II$ can satisfy these conditions along with satisfying the rules of the iteration game.  It is enough to show that there is some $b\in R$ satisfying all the clauses above except 2d. We can then choose $f$-least such $b$.\\

\textit{Claim.} There is $b\in R$ satisfying 1 and 2a-2c.\\
\begin{proof}
Fix some perfect hull $(M, \pi)$ at $\mu_{k+1}$ such that $(R, \tau)<(M, \pi)$.\\

\textit{Subclaim.} $\T$ is according to $\pi_{R, M}$-pullback of $\Sigma^M_{\mu_{k+1}, \P_k^M}$.\\
\begin{proof}
Fix limit $\a<lh(\T)$. By induction, we assume that $\T\rest \a$ is according to $\pi_{R, M}$-pullback of $\Sigma^M$. Notice that $\pi_{R, M}\T\rest \a\in M$. To prove the subclaim, we need to see that if $b$ is the branch of $\T\rest\a$ then $b$ is also the branch of $\pi_{R, M}\T\rest \a$. Notice that we must have that $\T\rest \a$ isn't maximal and hence, $\Q(\T\rest \a)$ exists and $\Q(\T\rest \a)\insegeq \M^\T_\a$. Let $c=\Sigma^M_{\mu_{k+1}, \P_k^M}(\pi_{R, M}\T\rest \a)$. It follows that
\begin{center}
$\Q(c, \pi_{R, M}\T\rest \a)$ exists iff $\Q(c, \T\rest \a)$ exists 
\end{center} 
implying that if $b\not =c$ then $\Q(c, \T\rest \a)$ doesn't exist. We then have that $\pi_{R, M}\T\rest \a$ is a maximal tree. Let $\U=\pi_{R, M}\T\rest \a$. Because $\U\in M$, it follows from the construction of $\Sigma^M_{\mu_{k+1}, \P_k^M}$ that there is a map $\sigma:\M^\U_c\rightarrow \P$ such that $\pi\rest \P^M=\sigma\circ i^\U_c$. Let $k:\M^{\T\rest \a}_c\rightarrow \M^\U_c$ be the copy map. It then follows that $\tau\rest \P^R=\sigma\circ k\circ i^{\T\rest \a}_c$ which implies, using \rlem{realizability gives fullness}, that $\M^{\T\rest \a}_c$ is $\omega$-suitable. Hence, $\Q(\T\rest \a)$ cannot exist, contradiction! 
\end{proof}

We then let $b=\Sigma^M_{\mu_{k+1}, \P_k^M}(\pi_{R, M}\T)$. It follows from the proof of the subclaim that $b$ satisfies clause 1. Suppose then $\T$ is maximal. It follows from the copying construction that $\pi_{R, M}\T$ is also maximal. It also follows from the construction of $\Sigma^M_{\mu_{k+1}, \P_k^M}$ that there is $\sigma:\M^{\pi_{R, M}\T}_b\rightarrow \P$ such that $\pi\rest \P^M=\sigma\circ i^{\pi_{R, M}\T}_b$. Letting $l:\M^\T_b\rightarrow \M^{\pi_{R, M}\T}_b$ come from the copying construction, we have that $\tau\rest \P^R=\sigma\circ l \circ i^\T_b$. Hence, $\M^\T_b$ is $\omega$-suitable. Moreover, $b$ satisfies clause 2b-2c. To finish, we need to show that $b\in R[h_0*g_0]$. Recall that $R[h_0*g_0]$ is $\omega$-closed in $W[g_0]$ and that in $W[g_0]$, $\cf(\d_{k+1}^{\P^R})=\omega$. Let then $\la \gg_i: i<\omega\ra\subseteq rng(i^\T_b)$ be cofinal in $\d(\T)$. Then we have that $\la \gg_i: i<\omega\ra\in R$. We also have that $b$ is the unique branch $c$ of $\T$ such that $\la \gg_i: i<\omega\ra\subseteq i^\T_c$. It then follows from absoluteness that $b\in R$.
\end{proof}  

We then define $\Lambda(\T)=b$ which satisfies clause 1 and 2 above. Suppose now that we have defined $\Lambda$ for stacks that have $<\a$-many rounds. We want to extend $\Lambda$ to act on stacks that have $\a$-many rounds. The first step is to specify $II$'s move at the beginning of round $\a$. To do this, we examine three cases. Suppose first $\a$ is limit. Let $\la \T_\xi, \M_\xi, \pi_{\xi, \gg}: \xi<\gg<\a\ra$ be a run of $\mathcal{G}^a(\P^R, \omega^3, \k_R)$ which is according to $\Lambda$. Then $I$ must start the $\a$th round on $\M_\a$ which is the direct limit of $\M_\xi$'s under the iteration embeddings $i_{\xi, \gg}:\M_\xi\rightarrow \M_\gg$. Let $i_{\xi, \a}:\M_\xi\rightarrow \M_\a$ be the direct limit embedding. We then let $\pi_\a:\M_\a\rightarrow \P$ come from the direct limit construction: given $x\in \M_\a$, let $\xi$ be such that for some $\bar{x}\in \M_\xi$, $x=i_{\xi, \a}(\bar{x})$ and let
\begin{center}
$\pi_{\a}(x)=_{def}\pi_\xi(\bar{x})$.
\end{center} 
We then let $II$ play $\pi_\a$. It follows that $\pi_\a$ is as desired and that $\M_\a$ is $\omega$-suitable.

Next suppose $\a=\b+1$ and let $\la \T_\xi, \M_\xi, \pi_{\xi, \gg}: \xi<\gg\leq\b\ra$ be a run of $\mathcal{G}^a(\P^R, \omega^3, \k_R)$ which is according to $\Lambda$ and suppose $I$ wants to start the $\a$th round of the game. Then $I$ has to start the round on the last model $\M_\a$ of $\T_\b$. It follows from our inductive hypothesis that $\M_\a$ is $\omega$-suitable. It remains to choose $\pi_\a$. Suppose first that $\T_\b^-$ is defined and is maximal. Let $b$ be the last branch of $\T$. As part of our inductive hypothesis we have that there is $\pi:\M^{\T_\b^-}_b\rightarrow \P$ such that $\pi_\b=\pi\circ i^{\T^-}_b$. In this case, we let $\pi_\a=\pi$. Next suppose that either $\T_\b^-$ isn't defined or that it is defined but it is short. Let $k<\omega$ be such that $\T_\b$ is a tree played on the window $(\d_k^{\M_\b}, \d_{k+1}^{\M_\b})$. Also, let $(M,\pi)>(R, \tau)$ be a perfect hull at $\mu_{k+1}$. It follows from the proof of the subclaim that $\T_\b$ is according to $\pi_{R, M}$-pullback of $\Sigma^M_{\mu_{k+1}, \P_k^M}$. Let then $l:\M_\a\rightarrow \N$ come from the copying construction where $\N$ is the last model of $\pi_{R, M}\T_\b$. By the construction of $\Sigma^M_{\mu_{k+1}, \P_k^M}$ we have $\sigma:\N\rightarrow \P$ such that $\pi\rest \P^M=\sigma\circ i^{\pi_{R, M}\T_\b}$. It then follows that $\pi_\b=\sigma\circ l \circ i^{\T_\b}$. Let then $\pi_\a=\sigma\circ l$. 

Next, to describe $II$'s moves in the $\a$th round we just follow the steps describing $II$'s move in the first round. More precisely, given a tree $\T\in R$ on $\M_\a$ which is according to $\Lambda$ and has limit length, we let $II$ play as follows:
\begin{enumerate}
\item If $\T$ is short then $II$ plays the unique cofinal branch $b\in R$ of $\T$ such that $\Q(b, \T)\insegeq \W(\M(\T))$ exists,
\item if $\T$ is maximal then $II$ plays $b$ such that
\begin{enumerate}
\item $b\in R$,
\item $\M^\T_b$ is $\omega$-suitable,
\item there is $\sigma:\M^\T_b\rightarrow \P$ such that $\pi_\a=\sigma\circ i^\T_b$.
\item $b$ is the $f$-least branch satisfying 2a-2c.
\end{enumerate}
\end{enumerate}
That there is such a branch $b\in R$ follows from the proof of the claim above. This finishes our description of $\Lambda$.
\end{proof}

We let $\Lambda^{R}$ be the winning strategy of $\P^R$ in the iteration game $\mathcal{G}^a(\P^R, \omega^3, \k_R)$ constructed above. Next, we show that perfect hulls of $\P$ are iterable. More precisely, suppose $(M^*, \pi^*)< (R, \tau)$ is a perfect hull at $\mu_p$ for some $p<\omega$. Let $\S\in R$ be an $\omega$-suitable $\Lambda^R$ iterate of $\P^R$ and let $k: \S\rightarrow \P$ be the realization according to $\Lambda$. Let $(M, \pi)>(R, \tau)$ be a perfect hull at $\mu_{p+1}$ such that $k\in rng(\pi)$. Notice that $H_{\mu_p^{++}}\in M$. We let $\Sigma^{R, M, \S}$ be the $k^*$-pullback of $\Lambda^M$ where $k^*=\pi^{-1}(k)$. Because $H_{\mu_p^{++}}\in M$, we have that $\Lambda^{R, M, \S}$ is a $(0, \omega^3, \mu_p^{++})$-strategy for $\S$ acting on stacks above $\d_{p}^\S$. 

\begin{lemma} $\P=\W(\P^-_\omega)$. 
\end{lemma}
\begin{proof}  Suppose not and let $k$ be least such that $\rho_\omega(\P)<\d_k$. Suppose $(R, \tau)$ is a perfect hull at $\mu_{k+1}$. Let $\Q=\P^{R}$ and let $p$ be the standard parameter of $\Q$. Let $\N=Hull_m^\Q(\d_{k+1}, \{p\})$ where  $m$ is the least such that $\rho_m(\Q)<\d_\omega^\Q$. We claim that $\Sigma^R_{\mu_{k+1}, \P^R_k}$ has branch condensation. To see this, it is enough to show that if $\VT$ is a stack on $\N$ according to $\Sigma^R$ such that $i^\VT$-exists and $\S$ is the last model of $\VT$ then $\S$ is $(m, \mu_{k+1}^{++})$-iterable above $\d^\S_k$. But choosing $M$ as in the paragraph proceeding the lemma, we see that $\Lambda^{R, M, \S}$ is such a strategy for $\S$.

It then follows that $\Sigma=_{def}\tau(\Sigma^R_{\mu_{k+1}, \P^R_k})$ is a fullness preserving $(\k, \k)$-strategy for $\P_{k+1}$ acting on trees above $\d_k$. Hence,  $\Sigma$ can be extended to a $(\k^+, \k^+)$-iteration strategy. Let then $h\subseteq Coll(\omega, \mu_{k+2})$ be $W$-generic. It then follows from \rthm{strategies are determined} that in $W[g_0*h]$, $L^{\Sigma}(\bR)\models AD^+$ implying that letting $\Sigma^*=_{def}\Sigma\rest H_{\omega_1}^{W[g_0*h]}$, $\Sigma^*\in S_{\mu_{k+2}, h_0*g_0*h}$. Hence, in $W[g_0*h]$, $L(\Sigma^*, \bR)\models AD^++\theta_0<\Theta$, contradiction!
\end{proof}

\subsection{A derived model at $\k$}

Our goal now is to show that if $G\subseteq Coll(\omega, \k)$ is $W[g_0]$-generic and $g_i=G\cap Coll(\omega, \mu_i)$ then $S_{\mu_i, h_0*g_0*g_i}$'s can be dovetailed into one model of determinacy. In the next subsection, we will use this fact to construct a strategy for some $\P_\mu$ which has branch condensation. Our first lemma is a strengthening of \rlem{fullness of M at A} and \rlem{correctness of good hulls}.

\begin{lemma}\label{correctness of good hulls all the way} Suppose $(M, \pi)$ is an excellent hull at some excellent $\mu$. Let $g\subseteq Coll(\omega, \mu)$ be $W$-generic and let $B\subseteq \k_M$ be a set in $M[h_0*g]$. Then $(\W(B))^{M[h_0*g]}=\W(B)$.
\end{lemma}
\begin{proof} We clearly have that $(\W(B))^{M[h_0*g]}\insegeq \W(B)$. We need to show that $\W(B)\insegeq (\W(B))^{M[h_0*g]}$. Suppose not. Let then $\M\insegeq \W(B)$ be the least such that $\rho_\omega(\M)=B$ and $(\W(B))^{M[h_0*g]}\inseg \M$. To derive contradiction, we need to show that $\M\in M[h_0*g]$ and $M[h_0*g]\models ``\M$ is $\k^+$-iterable". Let $\Q=\P^M_\mu$. Recall that $\Q$ is short tree iterable in $S_{\mu, h_0*g}$. 

Let $C\in M$ be the name for $B$. We claim that if $\T$ is the correctly guided tree on $\Q$ for making $C$ genericaly generic then $\T\in M$. To see this, it is enough to show that if $\a<lh(\T)$ is limit then $\Q(\T\rest \a)\in M$. It will then follow that $M\models ``\Q(\T\rest \a)$ is countably iterable" and hence, $M$ can correctly identify the branch of $\T\rest \a$. Suppose then $\Q(\T\rest \a)\not \in M$. Let $\S=(Lp(\T\rest \a))^M$. \\

\textit{Claim.} There is a cofinal wellfounded branch $c$ of $\T\rest \a$ such that $c\in M$.\\
\begin{proof} Let $\T^*=\pi(\T\rest \a)$ and $\S^*=\pi(\S)$. Notice that $\cf^V(\d(\T^*))<\k$. Let $\l=\cf^V(\d(\T^*))$. Fix an excellent $\eta>\mu$ such that $\eta^\l=\eta$. Let $(N, \sigma)>(M, \pi)$ be an excellent hull at $\eta$ such that $N^\l\subseteq N$. Let $\U=\sigma^{-1}(\T^*)$. We have that $\U$ is correctly guided and hence, $\U$ is according to $\pi(\Sigma^M_{\mu, \vec{B}})$. Let then $b=\pi(\Sigma^M_{\mu, \vec{B}})(\U)$. Let $\xi\in b$ be least such that $i^\U_{\xi, b}$ is defined. Because $\cf^V(\d(\U))=\l$, we can fix a cofinal $Y\subseteq rng(i^\U_{\xi, b})$ of order type $\l$. We have that $Y\in N$ and $b$ is the unique branch $c$ such that for some $\gg\in c$, $Y\subseteq rng(i^{\U}_c)$. It then follows that $b\in N$. By elementarity, there is a wellfounded branch $c$ of $\T\rest \a$ such that $c\in M$. 
\end{proof}

Let then $c$ be as in the claim. If $c=\pi(\Sigma^M_{\mu, \vec{B}})(\T\rest \a)$ then we are done as then $\Q(c, \T\rest \a)$ exists. Suppose then 
$c\not =\pi(\Sigma^M_{\mu, \vec{B}})(\T\rest \a)$. It then follows that $\T\rest \a$ has two well-founded branches implying that $\cf(\d(\T\rest \a))=\omega$. Let $b=\pi(\Sigma^M_{\mu, \vec{B}})(\T\rest \a)$ and let $\xi\in b$ be the least such that $i^{\T\rest \a}_{\xi, b}$-exists. Again, by choosing an $\omega$-sequence cofinal in $\d(\T\rest \a)$ and repeating the proof of the claim using the fact that $M$ is $\omega$-closed, we get that $b\in M$. It then follows that, after all, $\Q(\T\rest \a)\in M$, contradiction!

Let now $\T\in M$ be the correctly guided tree on $\Q$ for making $C$-genericaly generic. Let now $b=\pi(\Sigma^M_{\mu, \vec{B}})(\T)$. Repeating the above arguments we can show that $b\in M$. Let then $\S=\M^\T_b$. We have that $B$ is generic over $\S$ for the extender algebra of $\S$. Moreover, because $\S$ is full, we have that $\W(B)\in \S[B]$. Hence, $\M\in M$.

To show that $\M$ is $\k^+$-iterable in $M[h_0*g]$ we repeat the above argument and the proof of \rlem{correctness of good hulls}. Let $\Lambda$ be the iteration strategy of $\M$. Given a tree $\U\in M$ on $\M$ according to $\Lambda$ and of limit length, by the above proof we can find a $\pi(\Sigma^M_{\mu, \vec{B}})$-iterate $\S$ of $\Q$ such that $(\M, \U)$ are generic over $\S$. It then follows that $\Lambda(\U)\in \S[\M, \U]$ implying that if $b=\Lambda(\U)$ then $b\in M[h_0*g]$. $b$ can then be identified in $M[h_0*g]$ as the unique branch of $\U$ such that the phalanx $\Phi(\U^\frown \M^\U_b)$ is countably iterable. 
\end{proof}

\begin{lemma}\label{realizability gives fullness all the way} Suppose $(M, \pi)$ is a perfect hull at $\mu_m$ for some $m$ and let $\Q=\P^M$. Suppose $\sigma:\Q\rightarrow_{\Sigma_1} \S$ is such that $\S\in (H_{\k_M^+})^M$ and there is $k:\S\rightarrow \P$ such that $\pi\rest \Q=k\circ \sigma$. Then $\S=\W(\S|\d_\omega^\S)$.
\end{lemma}
\begin{proof} Let $T=T_{\mu_{m+1}}$. Let $B\in \powerset(\k_M)\cap M$ code $\Q, \S$. Let $\eta=\k_M$. We then have that $L[T, A_M, B]\models H_{\eta^{++}}=\W(\W(A_M, B))$. Let $\xi=\k^{+}$ and let $k: N\rightarrow L_\xi[T, A_M, B]$ be an elementary such that $L[T, A_M, B]\models ``\card{N}=\eta^+ \wedge \cp(k)>\eta^+"$ and let $S\in N$ be such that $k(S)=T$. Let $\nu=k^{-1}(\k^+)$. Notice that we have that $\powerset(\d_\omega^\Q)\cap L_\nu[S, \Q]=(\powerset(\d_\omega^\Q))^\Q$. It follows from \rlem{correctness of good hulls all the way} that $N\in M$

Suppose now that $\S\not=\W(\S|\d_\omega^\S)$. Let $\M\insegeq \W(\S|\d_\omega^\S)$ be least such that $\S\inseg \M$ and $\rho_{\omega}(\M)=\d_\omega^\S$. Notice that if $g\subseteq Coll(\omega, \S)$ is $W$-generic and $x\in L_\nu[S, \S][g]$ codes $\S$ then there is $y$ coding $\M$ such that $(x, y)\in p[S_n]$ where $n$ is as in \rlem{realizability gives fullness lemma}. This is because there is such a $y\in N[x]$. We can then repeat the proof of \rlem{realizability gives fullness lemma} to conclude that in fact $\M\insegeq \S$. 
\end{proof}

\begin{lemma}\label{derived model} Suppose $G\subseteq Coll(\omega, \k)$ is $W[g_0]$-generic and let $g_i=G\cap Coll(\omega, \mu_i)$. Let $\bR^*=\cup_{i<\omega}(\bR^{W[g_0*g_i]})$ and let $C_k=\cup_{i<\omega}B_{\mu_i, k, g_i}$. In $W(\bR^*)$, let $\Phi=\{ D\subseteq \bR^*: D$ is Wadge reducible to some $C_k\}$. Then $L(\Phi, \bR^*)\models AD^+$.
\end{lemma}
\begin{proof} It is enough to proof that the claim holds in some perfect $(M, \pi)$.  Fix then such a perfect $(M, \pi)$ at some $\mu_m$. Let $(N, \sigma)>(M, \pi)$ be a perfect hull at $\mu_{m+1}$ and let $\Lambda=\Lambda^{M, N, \P^M}$. Let $N_i=(H_{(\mu_i^M)^+})^M$. Let $\Q$ be the $\Lambda$-iterate of $\P^M$ which is obtained via $\la N_i: i<\omega\ra$-generic genericity iteration of $\P^M$. We have that if $i:\P^M\rightarrow \Q$ is the iteration embedding and $l: \Q\rightarrow \P$ is the last move of $II$ then $\pi\rest \P^M=l\circ i$ implying that $\Q=\W(\Q|\d_\omega^\Q)$ (this follows from \rlem{realizability gives fullness all the way}). Fix now some $G\subseteq Coll(\omega, \k_M)$ generic over $M[h_0*g_0]$ and let $g_i$'s be defined similarly as above. We define $C_k^M$ and $\Phi^M$ similarly working in $M[h_0*g_0*G]$. Let $\tau_k=\oplus_{i<\omega}\tau_{B_{\mu_m, k, G}}^{\Q|((\d_i^\Q)^{+\omega})^\Q}$. We have that $\tau_k\in \Q$ and therefore,  letting $\bR^*=\cup_{i<\omega}(\bR^{M[g_i]})$, we have that $B_{\mu_m, k, G}\cap \Q(\bR^*)=C_k^M$. It then follows from the derived model theorem that $L(\Phi^M, \bR^*)\models AD^+$.
\end{proof}

Given $G\subseteq Coll(\omega, \k)$, we let $\S_G=L(\Phi, \bR^*)$. Similarly, if $(M, \pi)$ is a perfect hull at some $\mu_p$ and $G\subseteq Coll(\omega, \k_M)$ is $M$-generic then we let $\S_G^M=(L(\Phi, \bR^*))^{M[G]}$.

\subsection{Strongly $\vec{B}$-guided strategy}

\begin{lemma}\label{strongly respecting strategy} Suppose $(R, \tau)$ is an excellent hull at $\mu_0$ and let $\R=\P^R$. Suppose $g\subseteq Coll(\omega, \mu_0)$ is generic and $p<\omega$. Then some tail of $\Sigma^{R}_{\mu_0, \vec{B}}$ strongly respects $B_p$. 
\end{lemma}
\begin{proof} Suppose not. Let $\Psi=\Sigma^{R}_{\mu_0, \vec{B}}$. Fix $\la \R_j, \Q_j, \sigma_j, \nu_j, \VT_j : j<\omega\ra\in M$ such that
\begin{enumerate}
\item $\R_0=\R$,
\item $\VT_j$ is a stack according to $\Psi$ on $\R_j$, $\R_{j+1}$ is its last model and $i^{\VT_j}$-exists,
\item $\sigma_j: \Q_j\rightarrow \R_{j+1}$ and $\nu_j: \R_j\rightarrow \Q_j$ are elementary
\item $i^{\VT_j}=\sigma_j\circ \nu_j$ and
\item $\sigma_j^{-1}(\tau^{\R_{j+1}}_{B_p})\not =\tau_{B_p}^{\Q_j}$. 
\end{enumerate}
We let $\R_\omega$ be the direct limit of $\R_j$'s under $\pi^{\VT_j}$'s. We have that there is $l:\R_\omega\rightarrow \P_\omega$ . Fix some perfect hull $(M, \pi)>(R, \tau)$ at $\mu_1$ and let $\Sigma=\Lambda^{R, M, \R_\omega}$. Let $\pi_{i, \omega}:\R_j\rightarrow \R_\omega$ be the direct limit embedding and let $\sigma_{j, \omega}=\pi_{j+1, \omega}\circ \sigma_j$. 

Next let $N_i=(H_{(\mu^R_i)^+})^R$ and let $h\subseteq Coll(\omega, \k_R)$ be $R$-generic such that $h\in W[g]$. Let $\sigma=\cup_{i<\omega}(\bR^{N_i[h\cap Coll(\omega, \mu^R_i)]})$. Let $\la \R^j_k, \Q^j_k, \VS^j_k, \VW^j_k, \sigma^j_k, \nu^j_k, m^j_k: j, k<\omega\ra$ be the iteration coming from simultaneous $\la N_i: i<\omega\ra$-generic genericity iteration of $\la \R_i, \Q_i: i<\omega\ra$ via $\Sigma$. Also, we let $\la \R^\omega_k, \Q^\omega_k, \VS^\omega_k, \VW^\omega_k, m^\omega_k, \sigma^\omega_k, \nu^\omega_k: k<\omega\ra$ be the direct limit of $\la \R^j_k, \Q^j_k, \VS^j_k, \VW^j_k, m^j_k, \sigma^j_k, \nu^j_k: j, k<\omega\ra$. 

It is not difficult to see that $\sigma$ is indeed the set of reals of some symmetric generic extension of each $\R^\omega_k$ and $\Q^\omega_k$. Moreover, letting $D_k$ and $C_k$ be the derived models of respectively $\R^\omega_k(\sigma)$ and $\Q^\omega_k(\sigma)$ we have that for every $k$,
\begin{center}
$\S^R_h\subseteq D_k\cap C_k$.
\end{center}
Notice that $B_p\cap \sigma \in \S^R_h$. Let then $s\in Ord^{<\omega}$ be such that $B_p\cap \sigma$ is definable from $s$ in $\S^R_h$. Let $\xi=\Theta^{\S^R_h}$. Notice then for each $k$, in $D_k$, $B_p\cap \sigma$ is the set definable from $s$ in $L(P_\xi(\bR))$. Let then $k$ be such that for all $n>k$,
\begin{center}
$m^\omega_n(s)=s$ and $m^\omega_n(\xi)=\xi$.
\end{center}
We then have that for every $n<\omega$,
\begin{center}
$\sigma^\omega_n(s, \xi)=\nu^\omega_n(s, \xi)=(s, \xi)$.
\end{center}
But this implies that for $n>k$, $\sigma^\omega_n(\tau_{B_p}^{\Q^\omega_n})=\tau_{B_p}^{\R^\omega_{n+1}}$. This then implies that for every $n>k$, $\sigma_n(\tau^{\Q_n}_{B_p})=\tau^{\R_{n+1}}_{B_p}$, contradiction. 
\end{proof}

Applying \rlem{strongly respecting strategy} repetitively and using the proof of \rlem{getting branch condensation}, we get that

\begin{corollary} Suppose $(R, \tau)$ is an excellent hull at $\mu_0$ and let $\R=\P^R_{\mu_0}$. Suppose $g\subseteq Coll(\omega, \mu_0)$ is generic. Then some tail of $\Sigma^{R}_{\mu_0, \vec{B}}$ has branch condensation. 
\end{corollary}

\subsection{Finishing the proof of \rthm{main theorem}}
Fix then $(R, \tau)$ as in the corollary above. It then follows that some tail of $\Sigma=_{def}\tau(\Sigma^R_{\mu_0, \vec{B}})$ has branch condensation. Let $\Q$ be the $\Sigma$-iterate of $\P_{\mu_0}$ such that letting $\Lambda$ be the corresponding tail of $\Sigma$, $\Lambda$ has a branch condensation. We then have that $\Lambda$ is fullness preserving, it has branch condensation and it is a $(\k, \k)$-iteration strategy. Appealing to \rlem{extending kappa strategies}, we can assume $\Lambda$ is a $(\k^+, \k^+)$-strategy. Let $\mu$ be an excellent point such that $\Q\in H_\mu$. It then follows from \rlem{strategies are determined} that if $g\subseteq Coll(\omega, \mu)$ is $W$-generic then in $W[g]$, $L^\Lambda(\bR)\models AD^+$. It then also follows that in $W[g]$, $L^\Lambda(\bR)\models AD^++\theta_0<\Theta$, contradiction!

\section{Summary}

As we mentioned before, the argument presented in the previous section does not use the full force of $\neg\square_\k$. This fact is only used in establishing that various $Lp$-closures of $V_\k$ have height $<\k^+$ and also in the proof of \rthm{strategies are determined} and \rthm{determinacy in the max model} (which were proved in \cite{PFA}). Below we state the exact hypothesis we need to carry out the proof of the conclusion of \rthm{main theorem}. To state the theorem, we need to introduce some terminology. 

\begin{definition}[Core model induction operator]\label{lightface operator} Suppose $\k$ is an uncountable cardinal. We say $F$ is a $\k$ core mode induction operator or just $\k$-cmi operator if one of the following holds:
\begin{enumerate}
\item For some countable $a$, $F$ is a $\k$ iteration strategy for some sound $a$-mouse $\M$ such that $\rho(\M)=a$. 
\item For some good pointclass $\Gamma$ such that $\Gamma\models MC$\footnote{i.e., for every $a\in HC$, $C_\Gamma(a)=Lp^\Gamma(a)$.} and some $\Gamma$-suitable $\M$, $F$ is a $\Gamma$-fullness preserving $(\k, \k)$-iteration strategy for $\M$ such that $F$ has branch condensation.
\item $F$ is a mouse operator total on $H_{\k}$.
\item  For some $G$ which satisfies 1, 2 or 3 above and some $n<\omega$, $F$ is $x\rightarrow \M^{\#, G}_n(x)$ operator defined on $H_\k$ or $F$ is the $(\k, \k)$-iteration strategy of $\M_n^{\#, G}$.
\end{enumerate}
\end{definition}

Suppose $\k$ is an uncountable cardinal and $F$ is a $\k$-cmi operator. For $a\in H_\k$, we let
\begin{enumerate}
\item $Lp^{F}(a)=\cup \{ \N : \N$ is a countably iterable\footnote{i.e., the countable submodels are $\omega_1+1$-iterable} sound $F$-mouse over $a$ such that $\rho(\N)=a\}$,
\item $\W^{\k, F}(a)=\cup \{ \N: \N$ is a $\k$-iterable sound $F$-mouse over $a$ such that $\rho(a)=a\}$,
\item $\K^{\k, F}(a)=\cup\{ \N: \N$ is a countably $\k$-iterable\footnote{i.e., the countable submodels are $\k$-iterable} sound $F$-mouse over $a$ such that $\rho(a)=a\}$.
\end{enumerate}
If $g$ is any $<\k$-generic and $F\in V[g]$ is a $\k$-cmi operator then we let $Lp^{F, g}$, $\W^{\k, F, g}$ and $\K^{\k, F, g}$  be defined similarly in $V[g]$. The following connects the three stacks defined above.

\begin{proposition}\label{facts on stacks} Suppose $\k$ is an uncountable cardinal and $F$ is a $\k$-cmi. Then for every $a\in H_\k$, $\W^{\k, F}(a)\insegeq \K^{\k, F}(a)\insegeq Lp^F(a)$. Suppose further that $F$ is $\k$-uB, $\eta<\k$ and $g\subseteq Coll(\omega, \eta)$ is $V$-generic. Then $\W^{\k, F, g}(a)\insegeq \W^{\k, F}(a)$, $\K^{\k, F, g}(a)\insegeq \K^{\k, F}(a)$ and $Lp^{F, g}(a)\insegeq Lp^{F}(a)$.
\end{proposition}

\begin{definition}[Covering by Lp, lpc] Suppose $\k$ is an uncountable cardinal. We then say that covering by lower part holds at $\k$ and write $lpc(\k)$ if for every $A\subseteq \k$, $\cf(o(Lp(A)))\geq \k$.
\end{definition}

\begin{definition}[The Derived Model Hypothesis, dmh] Suppose $\k$ is a strong limit cardinal. We say that the derived model hypothesis holds at $\k$ and write $dmh(\k)$ if for every countably closed cardinal $\mu<\k$ and for every $V$-generic $g\subseteq Coll(\omega, \mu)$ letting $\S_{\mu, g}=L(\K^{\k, g}(\bR^{V[g]}))$ we have 
\begin{enumerate}
\item $\S\models AD^+$ and
\item for every $\k$-cmi operator $F$, $\M_\omega^{\#, F}$-exists and is $\k$-iterable. 
\end{enumerate} 
\end{definition}

Below is the summary of what we proved in this paper. 

\begin{theorem}[Summary]\label{summary} Suppose $\k$ is a singular strong limit cardinal such that $\neg lpc(\k)$ and $dmh(\k)$ hold. Then for some countably closed $\mu<\k$, whenever $g\subseteq Coll(\omega, \mu)$ is $V$-generic, the following holds: 
\begin{enumerate}
\item For any $a\in H_{\omega_1}^{V[g]}$, letting $\Gamma=(\powerset(\powerset(\omega)))^{\S_{\mu, g}}$, 
\begin{center}
$\W^{\k, g}(a)= \W^{\Gamma}(a)$.
\end{center}
\item Suppose $\nu\in (\mu, \k)$ is also countably closed and let $h\subseteq Coll(\omega, \nu)$ be $V[g]$-generic. Suppose $a$ is countable in $V[g]$. Then 
\begin{center}
$\W^{\k, g}(a)=\W^{\k, g*h}(a)$.
\end{center} 
\item There is $(\P, \Sigma)\in V$ such that for any $\eta<\k$ and any $V$-generic $g\subseteq Coll(\omega, \eta)$
\begin{enumerate}
\item $V[g]\models ``\P$ is $\powerset(\powerset(\k))$-suitable",
\item $\Sigma$ is a $(\k^+, \k^+)$-iteration strategy for $\P$ such that in $V[g]$, there is a $(\k^+, \k^+)$-iteration strategy $\Sigma^+$ for $\P$ with the property that $\Sigma^+\rest V=\Sigma$ and $V[g]\models ``\Sigma^+$ has a branch condensation and is $\powerset(\powerset(\k))$-fullness preserving".
\end{enumerate}
\end{enumerate}
\end{theorem}

The proof of \rthm{summary} is exactly the argument presented in the previous sections. The facts, namely the conclusions of \rthm{strategies are determined} and \rthm{determinacy in the max model},  that were proved using the failure of $\square_\k$ are now part of our hypothesis. 

%

\bibliographystyle{plain}
\bibliography{square}
\end{document}